\let\SavedRefstepcounter\refstepcounter
\documentclass[11]{siamltex}
\let\refstepcounter\SavedRefstepcounter

\usepackage[pagewise]{lineno}

\newcommand{\nwc}{\newcommand}
\usepackage[pdftex]{graphicx}
\usepackage{nicefrac}
\usepackage{psfrag}
\usepackage{mathtools}
\usepackage{wrapfig}
\usepackage{epsf}
\usepackage{amsmath,amssymb}
\usepackage[font={footnotesize}]{caption} 
\usepackage{url}
\usepackage[mathscr]{euscript}
\usepackage{xcolor}
\usepackage{subcaption}
\usepackage{tikz}
\usetikzlibrary{patterns}
\usepackage{pgfplots}
\tikzset{every picture/.style={line width=0.75pt}} 
\usepackage{etoolbox}
\makeatletter
\patchcmd{\@begintheorem}{\scshape}
  {\MakeLinkTarget{\@currentcounter}\scshape}{}
  {\PackageError{siam-link-fix}{Could not patch theorem heading}{}}
\patchcmd{\@opargbegintheorem}{\scshape}
  {\MakeLinkTarget{\@currentcounter}\scshape}{}
  {\PackageError{siam-link-fix}{Could not patch optional theorem heading}{}}
\makeatother
\usepackage[hidelinks,hypertexnames=false]{hyperref}

\renewcommand{\theequation}{\arabic{equation}}

\newtheorem{thm}{Theorem}[section]

\newtheorem{rema}[thm]{Remark}

\newcommand{\diff}{\mathrm{d}}

\newcommand{\barint}{\hbox{$\int$\kern-0.75\intwidth
\vrule width 0.5\intwidth height 2.4pt depth -2pt\kern0.25\intwidth}}
\newlength\intwidth
\setbox0=\hbox{$\int$}
\intwidth=\wd0

\newcommand\avint{\hbox{\hbox{$\displaystyle \int$}\hbox{\kern-.9em{$-$}}}}

\newcommand\smavint{\hbox{\hbox{$\int$}\hbox{\kern-.75em{$-$}}}}

\nwc{\st}{^{\mbox{\it st}}}

\nwc{\veloc}{v}
\nwc{\rhoc}{\beta}
\nwc{\hl}{\hat{L}}

\newcommand{\bfx}{\boldsymbol{x}}
\newcommand{\bfP}{\boldsymbol{P}}

\def\Xint#1{\mathchoice
{\XXint\displaystyle\textstyle{#1}}%
{\XXint\textstyle\scriptstyle{#1}}%
{\XXint\scriptstyle\scriptscriptstyle{#1}}%
{\XXint\scriptscriptstyle\scriptscriptstyle{#1}}%
\!\int}
\def\XXint#1#2#3{{\setbox0=\hbox{$#1{#2#3}{\int}$}
\vcenter{\hbox{$#2#3$}}\kern-.51\wd0}}

\def\dashint{\Xint-}

\nwc{\intRp}{\int_0^\infty}
\nwc{\aint}{\dashint}
\nwc{\aaint}{\dashint}

\newcommand{\cA}{{\cal A}}

\newcommand{\cE}{{\cal E}}
\newcommand{\cF}{{\cal F}}
\newcommand{\cH}{{\cal H}}

\newcommand{\cL}{{\cal L}}

\newcommand{\cW}{{\cal W}}

\newcommand{\R}{\mathbb R}

\newcommand{\be}{\begin{eqnarray}}
\newcommand{\ee}{\end{eqnarray}}

\newcommand{\ben}{\begin{eqnarray*}}
\newcommand{\een}{\end{eqnarray*}}

\newcommand{\x}[2]{\|\; #1 \;\|_{#2}}

\newcommand{\half}{\nicefrac{1}{2}}

\DeclareMathOperator{\Tr}{Tr}

\newcommand\restr[2]{{
		\left.\kern-\nulldelimiterspace 
		#1 
		\littletaller 
		\right|_{#2} 
}}

\newcommand{\littletaller}{\mathchoice{\vphantom{\big|}}{}{}{}}

\title{}
\author{}
\date{}

\numberwithin{equation}{section}
\usepackage{indentfirst}
\usepackage{stmaryrd}

\usepackage{nameref}

\pgfplotsset{compat=1.18}
\begin{document}
	
\title{The $\alpha$-Limit Problem: Convergence of a Linear Degenerate Interface Transmission Problem}
\author{
	Toai Luong\thanks{Department of Mathematics and Statistics, Virginia Commonwealth University, Richmond, VA, USA. 
		Email: luongtt3@vcu.edu.
	} \and 
	Tadele Mengesha\thanks{Department of Mathematics, The University of Tennessee, Knoxville, TN, USA. 
		Email: mengesha@utk.edu. \textbf{Funding:} NSF-DMS 2509059.
	} \and 
	Kerrek Stinson\thanks{Department of Mathematics, The University of Utah, Salt Lake City, UT, USA. 
		Email: kerrek.stinson@utah.edu.
	} \and 
	Steven M. Wise\thanks{Department of Mathematics, The University of Tennessee, Knoxville, TN, USA. 
	Email: swise1@utk.edu. \textbf{Funding:} NSF-DMS 2309547 and 2607995.
	} \and 
	Ming Hei Wong\thanks{Department of Mathematics, The University of Tennessee, Knoxville, TN, USA. 
		Email: mwong4@vols.utk.edu. 
	}
}

\maketitle

    \begin{abstract}
We study the singular limit of a family of linear degenerate interface transmission problems arising from a regularization procedure in the newly proposed Two-Parameter Diffuse Domain Method (DDM2p). For $\alpha>0$, the regularized problem admits a strictly convex variational formulation on $H^{1}(\Omega)$. In the limit $\alpha\to0$, the problem degenerates to a weakly coupled interface system with a nonstandard energy structure. To characterize the limit, we introduce a closed Hilbert subspace $\mathcal{H}\subset H^{1}(\Omega)$, defined through an auxiliary Helmholtz problem on an annular subdomain $\Omega_2\subset \Omega$, and identify the limiting energy functional $\mathcal{E}_{0}$ on $\mathcal{H}$. We prove that the regularized energies $\mathcal{E}_{\alpha}$ {restricted to $\mathcal{H}$} $\Gamma$-converge to $\mathcal{E}_{0}$ in the strong $L^{2}(\Omega)$ topology, using the standard framework. Consequently, minimizers of $\mathcal{E}_{\alpha}$ converge to the unique minimizer of $\mathcal{E}_{0}$, which is shown to be equivalent to the solution of the limiting interface problem. We further prove a strong convergence $u_{\alpha}\to u_{0}$ in $H^{1}(\Omega)$ and establish a convergence rate of $O(\alpha)$. {Numerical experiments in one spatial dimension exhibit first-order convergence, suggesting that the proved $O(\alpha)$ upper bound is sharp.}
    \end{abstract}


    \section{Definition of the $\alpha$-Limit Problem}
For a given open domain {$\Omega_1$} with (at least) Lipschitz boundary $\Sigma := \partial \Omega_1$ and data $q \in [H^{1}(\Omega_1)]'$, the dual space of $H^{1}(\Omega_1)$, and $g \in H^{-\half}(\Sigma)$, the dual space of $H^{\half}(\Sigma)$, we are interested in finding a function $u_{0}^{(1)}:\Omega_1 \to \mathbb{R}$ that satisfies
    \begin{equation}
    \tag{$P_0 \, .1$}
    \label{bvp-OG} 
    \begin{aligned}
-\Delta u_{0}^{(1)} + \gamma u_{0}^{(1)} &= q, \quad &&\text{in } \Omega_1, 
    \\
{-\boldsymbol{n}_1\cdot  \nabla u_{0}^{(1)} } &= \kappa u_{0}^{(1)} + g, \quad &&\text{on } \Sigma = \partial\Omega_1.
    \end{aligned}
    \end{equation}
where $\gamma > 0$ and $\kappa \geq 0$ are given constants and {$\boldsymbol{n}_1$ is the outer normal on $\partial \Omega_1$.} The Robin boundary value problem (BVP) \eqref{bvp-OG} has a unique solution in $H^{1}(\Omega_1)$. Motivated by numerical methods for solving \eqref{bvp-OG} we embed this problem into a larger domain (see Section \ref{sec-motivate} below). Precisely, we {suppose $\Omega_1$ is compactly contained}
in a polygonal/polyhedral domain $\Omega$, as depicted in Figure~\ref{fig:domain}, {and find the elliptic extension $u_{0}^{(2)}$ of $u_{0}^{(1)}$ in $\Omega_2 = \Omega\setminus \overline{\Omega_{1}}$} satisfying 
    \begin{equation}
    \tag{$P_0 \, .2$}
    \label{bvp-OG-ext}
    \begin{aligned}
-\Delta u_0^{(2)} + \beta u_0^{(2)}  &= 0, \quad &&\text{in } \Omega_2, 
    \\
u_0^{(2)} &= u_0^{(1)}, \quad &&\text{on } \Sigma = \partial\Omega_1, \\
{\boldsymbol{n}_2} \cdot \nabla u_0^{(2)}   & = 0, \quad && \text{on } \partial\Omega, 
    \end{aligned}
    \end{equation}
where $\beta\geq 0$ is a given constant and {$\boldsymbol{n}_2$ is the outer normal on $\partial \Omega$}. The mixed BVP has a unique solution once $u_0^{(1)}$ {is defined.}

We will denote the combined problem \eqref{bvp-OG} -- \eqref{bvp-OG-ext} by Problem \hypertarget{po}{\hyperlink{po}{$(P_0)$}} and refer to it as a \emph{weakly-coupled, two-sided interface} problem.  The function 
    \[
u_0(\bfx) = 
    \begin{cases}
u_0^{(1)}(\bfx),  &\text{if } \bfx \in  \Omega_1, 
    \\
u_0^{(2)}(\bfx),  &\text{if } \bfx \in  \Omega_2
	\end{cases}
    \]
belongs to $H^{1}(\Omega)$ and we will show later in this paper that it can be obtained as a minimizer of a quadratic energy functional over a particular Hilbert space. We will also prove that the coupled problem, Problem \hyperlink{po}{$(P_0)$}, can be found as the limit of a degenerate transmission problem.

Next, we {introduce the (nearly) degenerate transmission problem, which may be regarded} as a regularized version of the last problem. For $\alpha > 0$, we consider the following two-sided $\alpha-$dependent boundary value problem in $\Omega$: Find a function $u_\alpha: \Omega \to \R$ defined as
    \begin{align*}
u_\alpha(\bfx) = 
    \begin{cases}
u_\alpha^{(1)}(\bfx),  &\text{if } \bfx \in  \Omega_1, 
    \\
u_\alpha^{(2)}(\bfx),  &\text{if } \bfx \in  \Omega_2,
	\end{cases}
    \end{align*}
where $u_\alpha^{(1)}:\Omega_1 \to \R$ and $u_\alpha^{(2)}:\Omega_2 \to \R$ satisfy
    \begin{equation}
    \tag{$P_\alpha$}
    \label{eqn:alpha-prob}
    \begin{aligned}
-\Delta u_\alpha^{(1)} + \gamma u_\alpha^{(1)} &= q, \quad && \text{in } \Omega_1,  
    \\
-\alpha\Delta u_\alpha^{(2)}  + \alpha\beta u_\alpha^{(2)} &= 0, \quad && \text{in } \Omega_2, 
    \\
u_\alpha^{(1)} &= u_\alpha^{(2)}, \quad && \text{on } \Sigma = \partial\Omega_1,  
    \\
- \boldsymbol{n}_1 \cdot  \nabla (u_\alpha^{(1)} - \alpha u_\alpha^{(2)}) & = \kappa u_\alpha^{(1)} + g, \quad && \text{on } \Sigma = \partial\Omega_1, 
	\\
\alpha \, {\boldsymbol{n}_2} \cdot \nabla u_\alpha^{(2)}   & = 0, \quad && \text{on }{\partial\Omega.}
    \end{aligned}
    \end{equation}
    {Notice that \eqref{eqn:alpha-prob} differs from \hyperlink{po}{$(P_0)$} in its Robin boundary condition, where the term $\alpha \nabla u_\alpha^{(2)}\cdot {\bf n}_1$ introduces a strong coupling; unlike in \hyperlink{po}{$(P_0)$}, $u_\alpha^{(1)}$ and $u_\alpha^{(2)}$ can no longer be computed independently.
Nevertheless, a key advantage of \eqref{eqn:alpha-prob} is that, as can be checked directly, its solution $u_\alpha$  corresponds} to a minimizer of the associated energy functional $\cE_\alpha: H^1(\Omega) \to \R$, defined by
    \begin{equation}
    \label{E_alpha-energy}
\cE_\alpha[u] = \int_\Omega   \frac{1}{2} \left( D_\alpha |\nabla u|^2 + c_\alpha u^2 \right)\mathrm{d}\bfx 	+ \int_{\Sigma} \frac{1}{2} \kappa u^2 {\mathrm{d}S} - \langle q_0,  u\rangle_{\Omega} + \langle g,u\rangle_{\Sigma},
    \end{equation}
for $u \in H^1(\Omega)$.  Here, $D_\alpha = \chi_1 + \alpha \chi_2$ and $c_\alpha = \gamma \chi_1 + \alpha\beta\chi_2$, 
where $\chi_i$ is the characteristic function of $\Omega_i$. {In \eqref{E_alpha-energy}, the volume bracket $\langle \,\cdot\, , \,\cdot\, \rangle_{\Omega}$  denotes duality with $H^1(\Omega)$ and the surface bracket $\langle \,\cdot\, , \,\cdot\, \rangle_{\Sigma}$ denotes duality with $H^{\half}(\Sigma)$. 
The function $q_0 \in [H^{1}(\Omega)]'$ is defined by
\begin{equation*}\label{q0-definition}
\langle q_0,v\rangle_\Omega:=\langle q,v|_{\Omega_1}\rangle_{\Omega_1},v\in H^1(\Omega).
\end{equation*}
This extension is well-defined since the restriction operator from $H^1(\Omega)$ to $H^1(\Omega_1)$ is bounded. If $q$ is an $L^2(\Omega_1)$ function, this definition of $q_0$ coincides with the usual zero extension $q\chi_1$. Motivated by this observation, we adopt the notation $q_0:=q\chi_1$ for general $q\in[H^1(\Omega_1)]'$, where the product is understood in the above duality sense.}
    
    \begin{figure}[htb!]
	\centering
	\begin{tikzpicture}[x=0.75pt,y=0.75pt,yscale=-1,xscale=1]
\draw  [line width=2.25]  (272,46) -- (423.67,46) -- (423.67,198.17) -- (272,198.17) -- cycle ;
\draw  [fill={rgb, 255:red, 184; green, 178; blue, 178 }  ,fill opacity=1 ] (380,107.33) .. controls (417,84.33) and (403,159.33) .. (348,163.33) .. controls (293,167.33) and (284,82.33) .. (314,96.33) .. controls (344,110.33) and (343,130.33) .. (380,107.33) -- cycle ;
\draw [line width=0.75]    (325.67,103) -- (338.16,81.59) ;
\draw [shift={(339.67,79)}, rotate = 120.26] [fill={rgb, 255:red, 0; green, 0; blue, 0 }  ][line width=0.08]  [draw opacity=0] (6.25,-3) -- (0,0) -- (6.25,3) -- cycle    ;
\draw [line width=0.75]    (424.17,120) -- (448.33,120.22) ;
\draw [shift={(451.33,120.25)}, rotate = 180.53] [fill={rgb, 255:red, 0; green, 0; blue, 0 }  ][line width=0.08]  [draw opacity=0] (6.25,-3) -- (0,0) -- (6.25,3) -- cycle    ;
		
\draw (337,130) node [anchor=north west][inner sep=0.75pt]    {$\Omega _{1}$};
\draw (290,165) node [anchor=north west][inner sep=0.75pt]    {$\Omega _{2}$};
\draw (342.33,68) node [anchor=north west][inner sep=0.75pt]    {$\boldsymbol{n}_{1}$};
\draw (454.83,115) node [anchor=north west][inner sep=0.75pt]    {$\boldsymbol{n}_{2}$};
		
	\end{tikzpicture}
\caption{We consider a bounded, open, simply connected domain $\Omega_1$ that is compactly contained in a bounded, open, polygonal domain $\Omega$. Set $\Omega_2 \coloneqq \Omega \setminus \overline{\Omega_1}$ and $\Sigma \coloneqq \partial \Omega_1$. Generalizations to $\R^n$ should be clear.}
	\label{fig:domain}
    \end{figure}
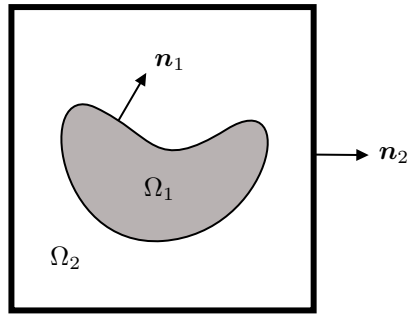

In this paper, our main goal is to show that $u_\alpha \to u_0$, as $\alpha \searrow 0$; the notions of convergence will be made precise later. In general terms, the collection of questions surrounding the convergence of solutions of \eqref{eqn:alpha-prob} to solutions of Problem~\hyperlink{po}{$(P_0)$} (i.e., \eqref{bvp-OG} -- \eqref{bvp-OG-ext}) is what we refer to as the \emph{$\alpha$-Limit Problem}.

We remark in passing that while we frame our results in terms of the constant elliptic coefficient $D_\alpha$, our results would naturally adapt to the case of elliptic variable coefficient $\mathbb{D}_\alpha = \mathbb{A}\chi_1 + \alpha \mathbb{B}\chi_2$, where $\mathbb{A}, \mathbb{B}\in L^\infty(\Omega;\R^{d\times d})$ are such that $\mathbb{A}(x)$ and $\mathbb{B}(x)$ are symmetric positive-definite with a spatially uniform lower-bound on the smallest eigenvalue.

This paper is organized as follows: in Section~\ref{sec-motivate} we motivate the $\alpha$-Limit Problem from a larger program focused on the diffuse domain method (DDM).  In Section~\ref{sec:Helmholtz-prob}, we explore the properties of the solution to a Helmholtz problem in an annulus with mixed boundary conditions. In Section~\ref{sec:cond-A},  we introduce a closed subspace of the Hilbert space $H^1(\Omega)$, denoted $\mathcal{H}$, over which the energy functional associated to Problem~{\hyperlink{po}{$(P_0)$}} is minimized. We also define the extension and restriction operators that will be used later in our analyses. In Section~\ref{sec:energy}, we define the energy functional, denoted $\cE_0$, associated to Problem~{\hyperlink{po}{$(P_0)$}}, and in Section~\ref{sec:Gamma-conv}, {after restricting $\cE_\alpha$ to $\cH$ and extending the resulting functionals by $+\infty$ to $L^{2}(\Omega)$, we define $\cF_\alpha$ and $\cF_0$ and prove $\cF_{\alpha} \overset{\Gamma}{\longrightarrow} \cF_{0}$.} 
While the $\Gamma$-convergence implies the convergence of minimizers with respect to the $L^2$-topology, in Section~\ref{sec:H1-conv}, we prove an alternative, improved convergence: $u_\alpha\to u_0$, as $\alpha\searrow 0$, strongly in {$H^1(\Omega)$. The
corresponding error has an $\alpha$ upper bound.} 
In Section~\ref{sec:numerical}, we present numerical experiments in one spatial dimension. The simulation results suggest that the predicted first-order convergence in $\alpha$ is sharp.

    \section{Motivation of the $\alpha$-Limit Problem}
    \label{sec-motivate}

The $\alpha$-Limit Problem, described in the previous section and the main focus of the paper, arises within the context of a larger theoretical program, namely, the convergence analysis of a diffuse domain method (DDM), which is, generally speaking, a method for approximating solutions to primarily elliptic problems on domains with complicated boundary shapes. 

The concept of DDM seems to have been first introduced {by Kockelkoren, Levine, and Rappel \cite{Kockelkoren-DDM2003}} to study diffusion within a cell with zero Neumann boundary conditions at the cell boundary.   Since then, DDMs have been applied to model electrical waves in the heart~\cite{Fenton-DDM2005};  membrane-bound Turing patterns~\cite{Levine-DDM2005}; deformation of bone~~\cite{aland2012adaptive}; tumor growth~\cite{Lowengrub-DDM2014}; grain boundary transport~\cite{Thornton-DDM2016}; osmosis~\cite{Ratz-DDM2016}; corrosion~\cite{Thornton-DDM2018}; fluid dynamics and fluid-structure interactions~\cite{Voigt-DDM2014, Voigt-DDM2010, Voigt-DDM2012, MR4642032,  Wise-DDM2021, Voigt-DDM2011-fluid, teigen2009diffuse}; and multicomponent vesicles~\cite{Voigt-DDM2006,Tang2025Vesicles}. Rigorous and asymptotic analyses of DDMs (and related methods) for solving elliptic PDEs in domains with complex boundaries subject to Dirichlet, Neumann, and Robin boundary conditions are provided in \cite{Abels-DDM2015, burger2015, burger2017, franz2012, lervag2015analysis, Li-DDM2009, schlottbom2016}. Rigorous numerical analysis of the DDM for solving parabolic problems has been carried out in~\cite{hao2026}. Very recently, Benfield and Dedner released the numerical software package DDFEM~\cite{benfield2025}, for solving diffuse domain problems using a finite element framework.

In this section,  we will introduce the DDM and explain how the $\alpha$-Limit Problem arises from within its analysis.  To do this, we refer back to Figure~\ref{fig:domain} repeatedly in this section. In any event, the $\alpha$-Limit Problem is an {interesting} problem in its own right and can be appreciated without any understanding of DDM.

    \medskip

\noindent\textbf{A Simple One-Sided Problem:} Suppose that our primary aim is to find a simple, practical approximation of the following one-sided problem: find the function $u:\Omega_1 \to \mathbb{R}$ that satisfies
\begin{equation} \label{bvp1-1-1}     
    \begin{aligned}
-\Delta u + \gamma u &= q, \quad &&\text{in } \Omega_1, 
    \\
{-\boldsymbol{n}_1\cdot  \nabla u } & = \kappa u + g, \quad &&\text{on } \Sigma = \partial\Omega_1.
    \end{aligned}
    \end{equation}
This problem is defined on the bounded, simply connected domain $\Omega_1$, and, of course, under very reasonable assumptions on the data, a unique solution is guaranteed. However, to make the problem interesting, we assume that the boundary, $\Sigma = \partial \Omega_1$, is very complicated, that is, not polygonal, circular, or otherwise simple.  Indeed, the main difficulty of this problem is the non-triviality of the boundary.

Of course, to find an approximate solution to this problem we could use a finite element method with a boundary-fitted mesh. In fact, one of the main selling points of the finite element method is that it can be used to approximate solutions on domains with highly complex (non-rectangular) shapes. However, the finite element method requires the construction of a very intricate triangulation of $\Omega_1$, which is a highly nontrivial algorithmic task. Furthermore, the mesh rarely ever fits exactly to the boundary, and, therefore, there is an extra source of error because of this misfit. 

There are alternatives to the finite element approximation method. Here, {we briefly introduce the {\it $2$-parameter diffuse domain method} (DDM2p) for approximating solutions to  {\eqref{bvp1-1-1}.}} In our construction, there are four steps in the creation of the diffuse domain approximation problem: (i) domain embedding, (ii) $\alpha$ regularization, (iii) singular reformulation, and (iv) $\varepsilon$ regularization. \emph{Our methodology is slightly different from the works referenced above, in that, to our knowledge, no other construction methods explicitly contain step (ii).} {It is for this reason that we highlight the second parameter $\alpha$ in the method's name. From the perspective of application, the inclusion of $\alpha$ is essential for computing solutions of the DDM (and is actually present in the numerics associated with prior analyses). With this clarification out of the way, we will simply refer to DDM2p by DDM when there is no confusion.}

    \medskip

\noindent\textbf{Step 1: Domain embedding:} Since $\Omega_1$ is bounded, naturally, we can find an open rectangular domain $\Omega$ that compactly contains $\Omega_1$. Consider the following problem: find a pair of functions, $u_{0,0}^{(1)}:\Omega_1 \to \mathbb{R}$ and $u_{0,0}^{(2)}:\Omega_2 \to \mathbb{R}$, that satisfy
    \begin{equation}
    \tag{$P_{0,0}$}
    \label{eqn:0-0-prob}
    \begin{aligned}
-\Delta u_{0,0}^{(1)} + \gamma u_{0,0}^{(1)} &= q, \quad &&\text{in } \Omega_1,  
    \\
-\Delta u_{0,0}^{(2)} + \beta u_{0,0}^{(2)} &= 0, \quad &&\text{in } \Omega_2, 
    \\
 u_{0,0}^{(1)} &= u_{0,0}^{(2)}, \quad &&\text{on } \Sigma = \partial\Omega_1,  
    \\
- \boldsymbol{n}_1\cdot  \nabla u_{0,0}^{(1)}  & = \kappa u_{0,0}^{(1)} + g, \quad &&\text{on }\Sigma = \partial\Omega_1, 
    \\
\boldsymbol{n}_2 \cdot \nabla u_{0,0}^{(2)}  & = 0, \quad &&\text{on } \partial\Omega.
    \end{aligned}
    \end{equation}
We write $u_{0,0} \coloneqq u_{0,0}^{(1)} \chi_1 + u_{0,0}^{(2)} \chi_2$. Problem~\eqref{eqn:0-0-prob} {subsumes} \eqref{bvp1-1-1}, though it now contains ancillary information, specifically information about the problem in $\Omega_2$. The interior problem on $\Omega_1$ is completely decoupled from the problem on the annular domain $\Omega_2$. The solution to {the interior problem,  $u_{0, 0}^{(1)}$}, is also the solution of \eqref{bvp1-1-1}. Once $u_{0,0}^{(1)}$ is found, $u_{0,0}^{(2)}$ can be computed as the annular Helmholtz function satisfying the Dirichlet boundary conditions $u_{0,0}^{(2)} = u_{0,0}^{(1)}$ on $\Sigma$, and homogeneous Neumann boundary conditions on $\partial\Omega$, the outer boundary. 

This problem has the advantage that it is equivalent to {\eqref{bvp1-1-1}} {and} it is posed on a rectangular composite domain $\Omega$, which can be covered by a uniform, rectangular mesh.  A clear disadvantage of Problem~\eqref{eqn:0-0-prob} is that it lacks 
{a single coercive energy formulation on $H^{1}(\Omega)$ of the type naturally suited to a standard conforming finite
 element treatment of the coupled problem. The two component problems nevertheless admit weak formulations on their respective subdomains.}

We point out that, in our framework, one can choose a non-rectangular domain (e.g., a polygonal/polyhedral domain) in which to embed $\Omega_1$ {when} there is a compelling reason to do so.
But a simple rectangular embedding suits our purposes for this discussion.

    \medskip

\noindent\textbf{Step 2: $\alpha$  Regularization:} 
 Now, let $\alpha\in(0,1)$ be given. Consider the following problem: find $u_{0,\alpha}^{(1)}:\Omega_1 \to \R$ and $u_{0,\alpha}^{(2)}:\Omega_2 \to \R$ that satisfy
    \begin{equation}
    \tag{$P_{0,\alpha}$}
    \label{eqn:alpha-0-prob}
    \begin{aligned}
{-\Delta u_{0,\alpha}^{(1)}} + \gamma u_{0,\alpha}^{(1)} &= q, \quad && \text{in } \Omega_1,  
    \\
-\alpha\Delta u_{0,\alpha}^{(2)}  + \alpha\beta u_{0,\alpha}^{(2)} &= 0, \quad && \text{in } \Omega_2, 
    \\
u_{0,\alpha}^{(1)} &= u_{0,\alpha}^{(2)}, \quad && \text{on } \Sigma = \partial\Omega_1,  
    \\
- \boldsymbol{n}_1 \cdot  \nabla (u_{0,\alpha}^{(1)} - \alpha u_{0,\alpha}^{(2)})  & = \kappa u_{0,\alpha}^{(1)} + g, \quad && \text{on } \Sigma = \partial\Omega_1, 
	\\
\alpha \, {\boldsymbol{n}_2} \cdot \nabla u_{0,\alpha}^{(2)}   & = 0, \quad && \text{on } \partial\Omega.
    \end{aligned}
    \end{equation}
We write $u_{0,\alpha} \coloneqq u_{0,\alpha}^{(1)} \chi_1 + u_{0,\alpha}^{(2)} \chi_2$. At least formally, we expect that $u_{0,\alpha} \to u_{0,0}$, as $\alpha\searrow 0$ in some topology. In fact, this convergence is the subject of the present paper. Note that problem~\eqref{eqn:alpha-0-prob} is equivalent to problem~\eqref{eqn:alpha-prob}. The need for the extra subscript in the former will become clear shortly.

Now, let us observe that the solution to Problem~\eqref{eqn:alpha-0-prob} minimizes the energy functional 
    \begin{align*}
\cE_{0,\alpha}[u] = \int_{\Omega}  \frac{1}{2}(D_{0,\alpha} |\nabla u|^2 + c_{0,\alpha} u^2)\mathrm{d}\bfx   + \int_{\Sigma}  \frac{1}{2}\kappa u^2 \mathrm{d} S - \langle {q_0}, u\rangle_{\Omega} +\langle g, u\rangle_{\Sigma},  
    \end{align*}
for all $u \in H^1(\Omega)$, where
    \begin{equation}
D_{0,\alpha}  \coloneqq\chi_1  + \alpha\chi_2 , \quad c_{0,\alpha}  \coloneqq \gamma\chi_1  + \alpha \beta\chi_2 , \quad {q_0}  \coloneqq q \chi_1   .	
    \label{eqn:d0-c0-f0}
    \end{equation}
Conversely, since $\cE_{0,\alpha}$ is coercive and strictly convex, it admits a unique minimizer $u_{0,\alpha}$ over $H^1(\Omega)$, which solves Problem \eqref{eqn:alpha-0-prob}. This is the sense in which the problem has been regularized: there is a standard energy minimization problem associated to Problem \eqref{eqn:alpha-0-prob}. This is not the case for Problem~\eqref{eqn:0-0-prob}, whose energy structure is quite non-standard, as we shall see. 

    \medskip

\noindent\textbf{Step 3: Singular Reformulation:} A solution to the two-sided problem \eqref{eqn:alpha-0-prob} can be understood in the weak sense. Again, let $\chi_i$ be the characteristic function for the domain $\Omega_i$, $i = 1, 2$, relative to the open set $\Omega$. If $(u_{0,\alpha}^{(1)}, u_{0,\alpha}^{(2)}) \in C^2(\overline{\Omega_1}) \times C^2(\overline{\Omega_2})$ is a solution pair to the two-sided problem \eqref{eqn:alpha-0-prob}, then $u_{0,\alpha} \coloneqq u_{0,\alpha}^{(1)} \chi_1 + u_{0,\alpha}^{(2)} \chi_2$ is in the space  $H^1(\Omega)\cap C^0(\overline{\Omega})$ and satisfies
    \begin{align} \label{weak-form}
\int_\Omega (D_{0,\alpha} \nabla u_{0,\alpha} \cdot \nabla w + c_{0,\alpha} u_{0,\alpha} w )\mathrm{d}\bfx  + \int_{\Sigma}\kappa u_{0,\alpha}w\, \mathrm{d}S - \langle {q_0}, w\rangle_{\Omega} +\langle g, w\rangle_{\Sigma} = 0,
    \end{align}
for any $w \in H^1(\Omega)$. The weak formulation suggests an equivalent singular formulation, which can be understood in the sense of distributions: find $u_{0,\alpha}\in H^1(\Omega)$ such that
    \begin{equation}
	\begin{split}
-\nabla \cdot (D_{0,\alpha} \nabla u_{0,\alpha}) + c_{0,\alpha} u_{0,\alpha} + (\kappa u_{0,\alpha} + g_{\mathrm{ext}} )\delta_\Sigma &= {q_0}, \quad \text{in } \Omega,  
    \\
\alpha \nabla u_{0,\alpha} \cdot \boldsymbol{n}_2 & = 0, \quad \text{on } \partial\Omega,
	\end{split}
    \tag{$P'_{0,\alpha}$}
    \label{eqn:alpha-0-prob-prime}
    \end{equation} 
where $\delta_\Sigma$ is the surface delta distribution {(equivalently, the Hausdorff surface measure)} with respect to the interface $\Sigma$ and $g_{\mathrm{ext}} \in [H^{1}(\Omega)]'$ is any distribution whose restriction on $\Sigma$ corresponds to $g$. {For sufficiently regular interfaces, such extensions exist.} For $g\in L^{p}(\Sigma)$, $1\leq p\leq \infty$, a simple extension is given in \cite[Section 2.3]{Abels-DDM2015} and several places elsewhere. {The extension is constructed as follows:} Suppose that $r(\bfx)$ is the signed distance function relative to $\Omega_1$, {which is assumed to be positive within $\Omega_1$ and negative outside $\overline{\Omega_1}$.} Define  
    \[
\bfP_{\Sigma}(\bfx) \coloneqq  \bfx - r(\bfx)\nabla r(\bfx),
    \]    
the orthogonal projection operator onto $\Sigma$. Then we set
    \[
g_{\mathrm{ext}} (\bfx) \coloneqq g(\bfP_{\Sigma}(\bfx)),
    \]
which is the constant extension of $g$ into a tubular neighborhood along the normal trajectories of $\Sigma$.  If $g\in H^{-\half}(\Sigma)$ is a distribution, its extension $g_{\mathrm{ext}}$ can be weakly defined by 
its action through duality on smooth test functions $\psi \in C_c^\infty(\Omega_\delta)$ within a narrow tubular neighborhood $\Omega_\delta = \{x \in \Omega : |r(\bfx)| < \delta\}$. {Once the extension $g_{ext}$ is determined, the distribution $g_{ext}\delta_\Sigma$ is defined directly by $\langle g_{ext}\delta_\Sigma,\psi\rangle_\Omega:=\langle g,\Tr_{\Sigma} (\psi)\rangle_\Sigma$, for every $\psi \in C_c^\infty(\Omega)$, where $\Tr_{\Sigma} : H^1(\Omega_1) \to  H^{\half}(\Sigma)$ is the trace operator. }

This singular reformulation is vital, because it is the basis for the DDM, as we show below.

    \medskip

\noindent\textbf{Step 4: $\varepsilon$ Regularization:} Problem \eqref{eqn:alpha-0-prob-prime} is still no easier to solve than the original problem. However, one more modification of the problem yields the diffuse domain approximation, which is quite straightforwardly solved. 

The pillars of the diffuse domain method (DDM) are approximations of the characteristic function $\chi_1$ and the surface delta function $\delta_\Sigma$ in the singular formulation of Problem \eqref{eqn:alpha-0-prob-prime}. A natural approximation for $\chi_1$ is the diffuse interface function
    \begin{align*}
\phi_\varepsilon(\bfx) \coloneqq \frac{1}{2}\left[ 1 + \tanh \left( \frac{r(\bfx)}{\varepsilon}\right) \right] \approx\chi_1(\bfx) = 
    \begin{cases}
1,  &\text{if } \bfx \in  \Omega_1, 
    \\
0,  &\text{if } \bfx \in  \Omega\setminus \Omega_1,
	\end{cases}
    \end{align*} 
where $\varepsilon > 0$ is a small parameter that controls the interface ``thickness", and $r(\bfx)$ is the signed distance function {as before.}  For simplicity, we assume that $\Sigma\coloneqq \partial\Omega_1$ is smooth so that the signed distance function, $r$, is well-defined and smooth in a neighborhood of the boundary.  Next, we use the approximation $|\nabla\phi_\varepsilon| \approx \delta_\Sigma$ for the surface delta function. The motivation for this approximation is the fact that $|\nabla\phi_\varepsilon|$ converges to $\delta_\Sigma$ in the sense of measures. Note that there are many possible choices for this approximation, see \cite{Abels-DDM2015}.

For each $\varepsilon,\alpha \in (0,1)$, the DDM approximation of problem \eqref{eqn:alpha-0-prob-prime} is given as follows: find a function $u_{\varepsilon,\alpha} : \Omega \to \mathbb{R}$ that satisfies
    \begin{equation}
	\begin{aligned}
-\nabla \cdot (D_{\varepsilon,\alpha} \nabla u_{\varepsilon,\alpha}) + c_{\varepsilon,\alpha} u_{\varepsilon,\alpha} + \left(\kappa u_{\varepsilon,\alpha}  + g_{\mathrm{ext}} \right)|\nabla\phi_\varepsilon| &= {q_{\varepsilon}}, \quad && \text{in } \Omega,  
    \\
D_{\varepsilon,\alpha} \nabla u_{\varepsilon,\alpha} \cdot \boldsymbol{n}_2 & = 0, \quad && \text{on } \partial\Omega,
    \end{aligned}
    \tag{$P_{\varepsilon,\alpha}$}
    \label{eqn:alpha-epsilon-prob}
    \end{equation}
where
    \begin{equation}
D_{\varepsilon,\alpha}  \coloneqq \alpha + (1 - \alpha)\phi_\varepsilon , \quad c_{\varepsilon,\alpha}  \coloneqq \alpha \beta + (\gamma - \alpha \beta)\phi_\varepsilon, \quad {q_{\varepsilon}}  \coloneqq   {q_0} \phi_\varepsilon, 
    \label{eqn:de-ce-fe}
    \end{equation}
and $g_{\mathrm{ext}}\in [H^{1}(\Omega)]'$ is as defined in the previous step.
Note that  $g_{\mathrm{ext}}|\nabla\phi_\varepsilon| \in [H^{1}(\Omega)]'$ with the property that 
    \[
\lim_{\varepsilon\to 0}\langle g_{\mathrm{ext}}|\nabla\phi_\varepsilon|, \psi \rangle_{\Omega} = \lim_{\varepsilon\to 0}\langle g_{\mathrm{ext}}, |\nabla\phi_\varepsilon| \psi \rangle_{\Omega} = \langle  g, \Tr_{\Sigma}(\psi)\rangle_{\Sigma},\quad \forall \, \psi\in H^{1}(\Omega).
    \]
For each pair $(\varepsilon,\alpha)$, a solution $u_{\varepsilon,\alpha}$ of problem \eqref{eqn:alpha-epsilon-prob}  minimizes the associated energy functional, $\cE_{\varepsilon,\alpha}$, defined by
    \begin{equation}
    \label{eqn:DDM-2Side-energy}
\cE_{\varepsilon,\alpha} [u] = \int_\Omega \left[ \frac{1}{2} (D_{\varepsilon,\alpha} |\nabla u|^2 + c_{\varepsilon,\alpha} u^2)  +  \frac{1}{2} \kappa u^2|\nabla\phi_\varepsilon|\right]\, \mathrm{d}\bfx +  {\langle g_{\mathrm{ext}}|\nabla\phi_\varepsilon|- {q_{\varepsilon}}, u\rangle_{\Omega}},  
    \end{equation}
for all $u \in H^1(\Omega)$. Since $\cE_{\varepsilon,\alpha}$ is coercive and strictly convex, it has a unique minimizer, which implies that the DDM problem \eqref{eqn:alpha-epsilon-prob} has a unique solution $u_{\varepsilon,\alpha} \in H^1(\Omega)$.

The major advantage of the DDM is that traditional finite element and finite difference codes can be used with no modification. The only requirement is that the signed distance function for the interface $\Sigma$ must be computed as a preprocessing step. {This can be accomplished by solving the Eikonal equation, for example,  by 
fast marching \cite{sethian1996}. }   
Structured, locally Cartesian finite element and finite difference grids can be used, and, consequently, fast geometric multigrid methods can be used to efficiently obtain numerical approximations. The downside of DDM is that the mesh must resolve the steep gradient in the diffuse interface function $\phi_\varepsilon$. However, adaptive, locally refined, multilevel meshes can mitigate this issue.

Note that, with our DDM construction, the choice of $\Omega$ is independent of $\varepsilon$ and $\alpha$ and, importantly, the stiffness matrices in the finite element and finite difference methods have diagonal elements that are bounded below by $\alpha$, in contrast to other methods that, without any intervention, have exponentially small diagonal elements.

\begin{figure}[t!]
    \centering
    \begin{tikzpicture}[x=2.5in,y=2.0in,line cap=round,line join=round]
        \useasboundingbox (-0.4,-0.3) rectangle (1.5,1.4);

        \draw[->, line width=0.9pt] (-0.04,0) -- (1.2,0) node[below] {$\varepsilon$};
        \draw[->, line width=0.9pt] (0,-0.04) -- (0,1.15) node[left] {$\alpha$};

        \draw[line width=0.8pt] (0,0) -- (1,0) -- (1,1) -- (0,1) -- cycle;

        \node[below left] at (0,0) {$u_{0,0}$};
        \node[above left] at (0,1) {$u_{0,\alpha}$};
        \node[above right] at (1,1) {$u_{\varepsilon,\alpha}$};
        \node[below right] at (1,0) {$u_{\varepsilon,0}$};

        \draw[->, line width=0.7pt] (0.9, 1.06) -- (0.1, 1.06);     
        \draw[->, line width=0.7pt] (-0.06, 0.9) -- (-0.06, 0.1);   
        \draw[->, line width=0.7pt] (1.06, 0.9) -- (1.06, 0.1);     
        \draw[->, line width=0.7pt] (0.9, -0.06) -- (0.1, -0.06);   

        \node[above, align=center, font=\small] at (0.5, 1.25) {
            $\mathcal{E}_{\varepsilon,\alpha} \xrightarrow{\Gamma} \mathcal{E}_{0,\alpha}$ 
        };
        \node[above, align=center, font=\small] at (0.5, 1.15) {
            $u_{\varepsilon,\alpha} \xrightarrow[]{\varepsilon \searrow 0} u_{0,\alpha}$
        };

        \node[above, align=center, font=\small] at (0.5, 1.08) {
(References~\cite{luong25b,luong25a})
        };

        \node[left, align=center, font=\small] at (-0.09, 0.70) {
            $\mathcal{E}_{0,\alpha} \xrightarrow{\Gamma} \mathcal{E}_{0,0}$ 
        };
         \node[left, align=center, font=\small] at (-0.09, 0.5) {
            $u_{0,\alpha} \xrightarrow[]{\alpha \searrow 0} u_{0,0}$ \\
            \rule{0pt}{3ex}$\alpha$-Limit Problem \\
            (this paper)
        };

        \node[right, align=center, font=\small] at (1.09, 0.5) {
            $u_{\varepsilon,\alpha} \xrightarrow{\alpha \searrow 0} u_{\varepsilon,0}$
        };
        \node[right, align=center, font=\small] at (1.09, 0.4) {
            (utility unknown)
        };

        \node[below, align=center, font=\small] at (0.5, -0.08) {
            (not practically \\ computable)
        };

        \node[align=center, font=\small] at (0.45, 0.70) {
            (future work)
        };

        \draw[line width=1pt] plot[domain=1:0.4, samples=50] (\x, {\x^2.2});
        \draw[->, line width=1pt] plot[domain=0.4:0, samples=50] (\x, {\x^2.2});

        \node[above, align=center, font=\small] at (0.45, 0.55) {
            $\mathcal{E}_{\varepsilon,\alpha(\varepsilon)} \xrightarrow{\Gamma} \mathcal{E}_{0,0}$ 
        };
        \node[above, align=center, font=\small] at (0.45, 0.45) {
           $u_{\varepsilon,\alpha(\varepsilon)} \xrightarrow[]{\varepsilon \searrow 0} u_{0,0}$
        };

        \node[right, font=\small] (alphadef) at (0.55, 0.23) {$\alpha(\varepsilon)$};

    \end{tikzpicture}
    \caption{Convergence diagram for the Two-Parameter Diffuse Domain Method (DDM2p). The refinement path of greatest utility for the DDM seems to be $(\varepsilon,\alpha(\varepsilon))$, where $\alpha :[0,1)\to [0,1)$ is smooth, positive, and satisfies $\alpha(0) = 0$, for example, $\alpha = \varepsilon^2$, as shown. The refinement path $(\varepsilon,\alpha = 0)$,  $\varepsilon\in (0,1)$, is not practical in DDM because the stiffness matrices in the finite element and finite difference approximations would have diagonal elements that are exponentially small in $\Omega_2$.}
    \label{fig:convergence-diagram}
\end{figure}
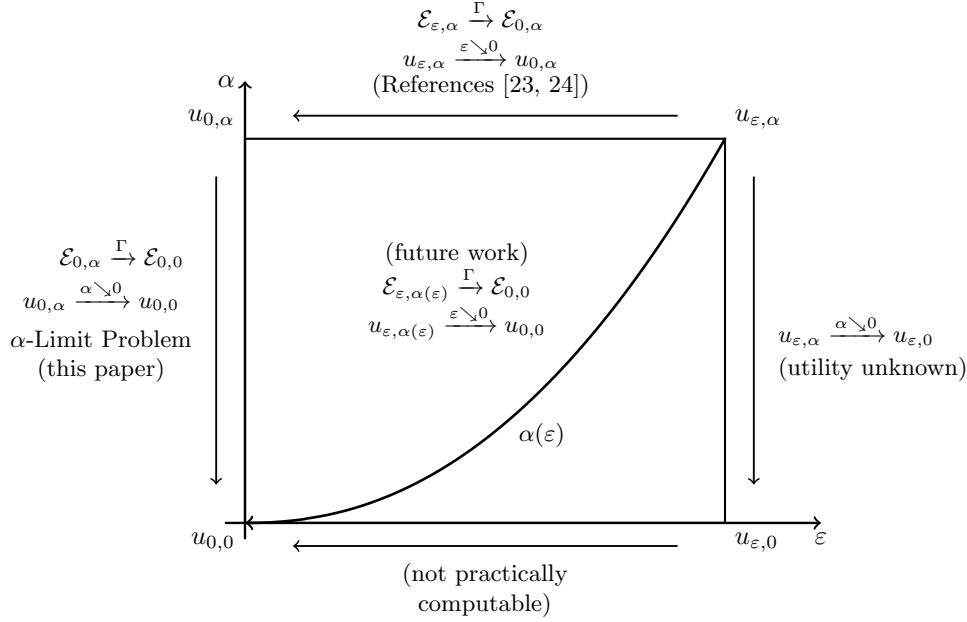

    \medskip

\noindent\textbf{Convergence Questions:}  Our constructions above lead us to consider several theoretical questions:
    \begin{enumerate}
    \item[Q1.]
Is there an energy, labeled $\cE_{0,0}$, say, associated to Problem~\eqref{eqn:0-0-prob}?

    \item[Q2.]
If so, is the following $\Gamma$-convergence valid: ${\cE_{0,\alpha}} \overset{\Gamma}{\longrightarrow} \cE_{0,0}$, as $\alpha \searrow 0$?  With respect to what topological space?
    \item[Q3.] 
Is the following convergence true: ${u_{0,\alpha}}\longrightarrow u_{0,0}$, as $\alpha\searrow 0$?  In what sense? At what order in $\alpha$?
    \end{enumerate}
See the convergence diagram in Figure~\ref{fig:convergence-diagram}. The {previous} three questions are the subject of this paper and are collectively referred to as the $\alpha$-Limit Problem. However, this is only one component in a larger program of inquiry.

In a series of two papers~\cite{luong25b,luong25a}, we proved a convergence result, ${\cE_{\varepsilon,\alpha}} \overset{\Gamma}{\longrightarrow} {\cE_{0,\alpha}}$,  for fixed $\alpha >0$, as $\varepsilon\searrow 0$. We also proved the convergence ${u_{\varepsilon,\alpha}}\longrightarrow {u_{0,\alpha}}$ in $H^1(\Omega)$, as $\varepsilon\searrow 0$, also for fixed $\alpha>0$. We were not able to theoretically establish rates of convergence. However, we proved asymptotic convergence at first-order in $\varepsilon$, which was supported by computations.

The broader program connected to the DDM seeks to answer the following:
    \begin{enumerate}
    \item[Q4.]
Is the following two-parameter $\Gamma$-convergence valid: ${\cE_{\varepsilon,\alpha}} \overset{\Gamma}{\longrightarrow} \cE_{0,0}$, as $(\varepsilon,\alpha) \to  (0,0)$? With respect to what topological space?
    \item[Q5.]
Suppose $p\in(0,\infty)$. Is the following one-parameter $\Gamma$-convergence valid: ${\cE_{\varepsilon,\varepsilon^p}} \overset{\Gamma}{\longrightarrow} \cE_{0,0}$, as $\varepsilon\searrow 0$? With respect to what topological space?
    \item[Q6.]
Is the following two-parameter convergence valid: ${u_{\varepsilon,\alpha}}\longrightarrow u_{0,0}$, as $(\varepsilon,\alpha)\to (0,0)$? In what sense? At what orders with respect to $\alpha$ and $\varepsilon$?
    \item[Q7.]
Is the following convergence valid: ${u_{\varepsilon,\varepsilon^p}} \longrightarrow u_{0,0}$, as $\varepsilon\searrow 0$? In what sense? At what order with respect to $\varepsilon$?
    \item[Q8.]
Is the following formal convergence valid: ${u_{\varepsilon,\varepsilon^p}}  \longrightarrow u_{0,0}$, asymptotically as $\varepsilon\searrow 0$? If so, at what order with respect to $\varepsilon$? This is a classic {\it corner layer} problem within the field of matched asymptotic analysis.
    \end{enumerate}

From the point of view of practical DDM computations, it is most natural to simultaneously decrease $\varepsilon$ and $\alpha$ to 0, along some predetermined refinement path. Therefore, the convergence questions Q5, Q7, and Q8 are of high importance. The exploration of the $\alpha$-Limit Problem is expected to shed some light on these and other questions regarding the convergence properties of DDM. \emph{Most importantly, perhaps, the present work defines the target energy $\cE_{0,0}$, which will be needed in all of the subsequent analyses of the Two-Parameter DDM.}

    \section{A Helmholtz Problem with Mixed Boundary Conditions}
    \label{sec:Helmholtz-prob}

In order to describe the proper energy structure of problem~{\hyperlink{po}{$(P_0)$}}, we need some facts about the Helmholtz problem in the annular domain $\Omega_2$. For the sake of generality, let us suppose that  $U \subset \R^n$ is a connected, open, bounded domain, such that $\partial U$ consists of two {disjoint sets} $\Gamma_1$ and $\Gamma_2$, each of positive $(n-1)$--dimensional Lebesgue measure. Assume that $\Gamma_1$ is relatively open and its relative closure is locally described by Lipschitz graphs. {The argument below uses the standard trace and mixed-boundary variational framework; see~\cite[Chapters 3--4]{McLean2000}.}

    \begin{proposition}
    \label{prop:Helmholtz}
For any given function $f \in H^{\half}(\Gamma_1)$,  the  problem
    \begin{equation}     
    \label{helmholtz-eqn}
	\begin{split}
-\Delta u +\beta u &= 0, \quad     \text{in } U, 
    \\
u &= f, \quad \text{on } \Gamma_1, 
    \\
\boldsymbol{n}  \cdot   \nabla u  &= 0, \quad \text{on } \Gamma_2,
    \end{split}
    \end{equation} 
has a unique weak solution $u_f \in H^1(U)$. {That is, there exists a unique $u_f\in H^{1}(U)$ such that $\Tr_{\Gamma_1}(u_f)=f$  and 
\[
\int_U(\nabla u_f\cdot\nabla v+\beta u_fv)\,\mathrm{d}\bfx=0
\]
for every $v\in \{v\in H^1(U):\Tr_{\Gamma_1}(v)=0\}$, where $\Tr_{\Gamma_1} : H^1(U) \to  H^{\half}(\Gamma_1)$ is the trace operator.}  Moreover, there exists a constant $C_0 > 0$, depending only on the domain $U$  and $\beta$, such that
    \begin{align*}
\|u_f\|_{H^1(U)} \leq C_0\|f\|_{H^{\half}(\Gamma_1)}.
    \end{align*}
    \end{proposition}

\begin{proof}
Since $f \in H^{\half}(\Gamma_1)$, there exists a function $\tilde{f} \in H^1(U)$ such that
\begin{align*}
	\Tr_{\Gamma_1}(\tilde{f}) = f, 
	\quad \text{and} \quad
	\|\tilde{f}\|_{H^1(U)} \leq C_1 \|f\|_{H^{\half}(\Gamma_1)},
\end{align*}
for some constant $C_1 > 0$ depending only on the domain $U$. 
Define the subspace
\begin{align*}
	V \coloneqq \left\{ v \in H^1(U) : \ \Tr_{\Gamma_1}(v) = 0 \text{ on } \Gamma_1 \right\}.
    \end{align*}
Then $u \in H^1(U)$ is a weak solution to Equation \eqref{helmholtz-eqn} if $u - \tilde{f} \in V$ and
\begin{align*}
	\int_U \nabla u \cdot \nabla v \, \mathrm{d}\bfx + \beta \int_U uv \, \mathrm{d}\bfx = 0,
	\quad \text{for all } v \in V.
\end{align*}
Let $w \coloneqq u - \tilde{f} \in V$, then the above weak formulation becomes
\begin{align*}
	\int_U \nabla w \cdot \nabla v \, \mathrm{d}\bfx + \beta \int_U wv \, \mathrm{d}\bfx
	= -\int_U \nabla \tilde{f} \cdot \nabla v \, \mathrm{d}\bfx - \beta \int_U \tilde{f} v \, \mathrm{d}\bfx,
	\quad \text{for all } v \in V.
\end{align*}
Let $L : V \to \R$ be the functional defined by
\begin{align*}
	L(v) = - \int_U \nabla \tilde{f} \cdot \nabla v \, \mathrm{d}\bfx - \beta \int_U \tilde{f} v \, \mathrm{d}\bfx,
\quad \text{for } v \in V,
\end{align*}
then $L$ is a bounded linear functional on $V$, with $\|L\| \leq \max\{1, \beta\} \|\tilde{f}\|_{H^{1}(U)}$. Define a bilinear form $A : V \times V \to \R$ by
\begin{align*}
	A(v_1, v_2)
	= \int_U \nabla v_1 \cdot \nabla v_2 \, \mathrm{d}\bfx + \beta \int_U v_1 v_2 \, \mathrm{d}\bfx,
	\quad \text{for } v_1, v_2 \in V,
\end{align*}
then $A$ is bounded. For any $v\in V$, since $\Tr_{\Gamma_1}(v) = 0$ on $\Gamma_1$, which is a positive-measure portion of $\partial U$, applying Poincar\'e's inequality, we get
    \begin{align*}
\|v\|_{H^1(U)}^2 \leq C_2 \|\nabla v\|_{L^2(U)}^2, 
    \end{align*}
for some constant $C_2 > 0$ depending only on the domain $U$.  Hence,
\begin{align*}
	A(v,v) = \int_U |\nabla v|^2 \, \mathrm{d}\bfx + \beta \int_U v^2 \, \mathrm{d}\bfx
	\geq \frac{1}{C_2} \|v\|_{H^1(U)}^2.
\end{align*}
So $A$ is coercive. 
By the Lax--Milgram Theorem, there exists a unique function $w_f \in V$ such that
\begin{align*}
	A(w_f,v) = L(v), \quad \text{for all } v \in V.
\end{align*}
Moreover, we obtain the estimate
    \begin{align*}
\|w_f\|_{H^1(U)} \leq \max\{1,\beta\} C_{2} \|\tilde{f}\|_{H^1(U)}.
    \end{align*}
Thus, the unique function $u_f \in H^1(U)$ satisfies 
$u_f - \tilde{f} = w_f \in V$ and
\begin{align*}
	\int_U \nabla u_f \cdot \nabla v \, \mathrm{d}\bfx  + \beta \int_U u_f v \, \mathrm{d}\bfx = 0,
	\quad \text{for all } v \in V.
\end{align*}
Furthermore,
  \begin{equation*}
  \begin{aligned}
\|u_f\|_{H^1(U)}
	&\leq (\max\{1,\beta\} C_{2} + 1) \|\tilde{f}\|_{H^1(U)}\\
	&\leq (\max\{1,\beta\} C_{2} + 1) C_1 \|f\|_{H^{\half}(\Gamma_1)}.
    \end{aligned}
    \end{equation*}
Define $C_0 \coloneqq (\max\{1,\beta\} C_{2} + 1) C_1$. Clearly, $C_0$ depends only on the domain $U$ and $\beta$, and
    \begin{align*}
\|u_f\|_{H^1(U)} \leq C_0\|f\|_{H^{\half}(\Gamma_1)},
    \end{align*}
as desired.
    \end{proof}
    
{We remark that the Neumann condition $\boldsymbol{n} \cdot \nabla u = 0$ on $\Gamma_2$ is a natural boundary condition, appearing explicitly only when the problem is written in strong form.}
    \section{The Space $\cH$}
    \label{sec:cond-A}
\subsection{Condition $\cA$ and the definition of $\cH$}
Let $\Omega_1$ be a bounded open subset of $\R^n$ {with Lipschitz boundary,} 
and let $\Omega$ be a larger cuboidal domain such that $\overline{\Omega_1} \subset \Omega$ and $\partial\Omega_1 \cap \partial\Omega = \varnothing$. 
Define $\Omega_2 \coloneqq \Omega \setminus \overline{\Omega_1}$. 
Let $\boldsymbol{n}_1$ denote the outward-pointing unit normal vector on $\Sigma = \partial\Omega_1$, and $\boldsymbol{n}_2$ denote the outward-pointing unit normal vector on $\partial\Omega$.  See Figure~\ref{fig:domain}. 
    \begin{definition}
For each $v \in H^1(\Omega)$, we say that $v$ satisfies \textbf{Condition $\cA$} if and only if {$v^{(2)} = \restr{v}{\Omega_2}$ is a weak solution of }
	\begin{align*}
-\Delta v^{(2)} + \beta v^{(2)} &= 0, \quad \text{in } \Omega_2, 
    \\
{\boldsymbol{n}_2 \cdot  \nabla v^{(2)}}  &= 0, \quad  \text{on } \partial\Omega. 
	\end{align*}
    \end{definition}
   {From the definition, if $v$ satisfies \textbf{Condition $\cA$},  then for every $\psi\in \{\psi\in H^1(\Omega_2):\Tr_{\Sigma}(\psi)=0\}, $ we have
    \[
\int_{\Omega_2}(\nabla v^{(2)}\cdot\nabla\psi+\beta v^{(2)}\psi)\,\mathrm{d}\bfx=0.
    \]}
    Define the set
    \begin{align*}
\cH \coloneqq \{ v \in H^1(\Omega) : \ v \text{ satisfies Condition $\cA$} \}.
    \end{align*}
{The following is a corollary of Proposition~\ref{prop:Helmholtz}. }
    \begin{corollary}\label{h(Omega_1)=scriptH}
   {A function $v^{(1)}$ belongs to $H^1(\Omega_1)$ if and only if there exists a unique function  $v\in \cH $  such that $v|_{\Omega_1}  = v^{(1)}$.  Moreover, there exist positive constants $C_1$ and $C_2$ such that 
   \[
   C_{1}\|v^{(1)}\|_{H^{1}(\Omega_1)}\leq \|v\|_{H^{1}(\Omega)} \leq C_{2}\|v^{(1)}\|_{H^{1}(\Omega_1)}
   \]}
    \end{corollary} 
    \begin{proof} 
Suppose that $v^{(1)} \in H^1(\Omega_1)$. Invoking Proposition~\ref{prop:Helmholtz}, let $v^{(2)} \in H^1(\Omega_2)$  be the unique solution to  
    \begin{equation}
    \begin{aligned}\label{eq:v2-from-v1}
-\Delta v^{(2)} + \beta v^{(2)} &= 0, \quad \text{in } \Omega_2, 
    \\
v^{(2)} &= v^{(1)}, \quad  \text{on } \Sigma = \partial\Omega_1, 
    \\
\boldsymbol{n}_2 \cdot 	\nabla v^{(2)}  &= 0, \quad  \text{on } \partial\Omega.
    \end{aligned}
    \end{equation}
   {Moreover, there exists a positive constant $C$ such that  
    \begin{equation}\label{normv2=normv1}
    \|v^{(2)}\|_{H^{1}(\Omega_2)}  \leq C \|\Tr_{\Sigma}(v^{(1)})\|_{H^{\half}(\Sigma)} \leq C\|v^{(1)}\|_{H^{1}(\Omega_1)},  \end{equation}
    where the last inequality follows from the boundedness of the trace operator. 
Then the function $v$ defined by
    \begin{equation}
v(\bfx) = 
    \begin{cases}
v^{(1)}(\bfx),  &\text{if } \bfx \in  \Omega_1,  
    \\
v^{(2)}(\bfx),  &\text{if } \bfx \in  \Omega_2,
	\end{cases}
    \end{equation}
    satisfies Condition $\mathcal{A}$, and thus 
belongs to $\cH$. Moreover, using \eqref{normv2=normv1}, 
\[
\|v^{(1)}\|_{H^{1}(\Omega_1)} \leq \|v\|_{H^1(\Omega)} \leq \|v^{(1)}\|_{H^{1}(\Omega_1)} + \|v^{(2)}\|_{H^{1}(\Omega_2)} \leq C \|v^{(1)}\|_{H^{1}(\Omega_1)}. 
\]}
Uniqueness follows since any $\tilde v\in \mathcal{H}$ with $\tilde v = v^{(1)}$ in $\Omega_1$ must clearly satisfy \eqref{eq:v2-from-v1} in $\Omega_2$. The converse is trivial.
    \end{proof}

{The preceding result motivates the following proposition. } 
    \begin{proposition} \label{H-is-Hilbert}
Define the bilinear form  $(\,\cdot\, , \,\cdot\,)_{\cH}$  via  
	\begin{align*}
(v,w)_{\cH} \coloneqq \left(v^{(1)},w^{(1)}\right)_{H^1(\Omega_1)}, \quad  \forall \, v,w \in \cH,
	\end{align*}
where $v^{(1)} \coloneqq \restr{v}{\Omega_1}$ and $w^{(1)} \coloneqq \restr{w}{\Omega_1}$. Then $\left(\cH, (\,\cdot\, , \,\cdot\,)_{\cH} \right)$ is a Hilbert space. The associated norm is defined as usual, by $\|v\|_{\cH} \coloneqq \sqrt{(v,v)_{\cH}}$, for all $v\in\mathcal{H}$.
    \end{proposition}

    \begin{proof}
That $\cH$ is a linear space is easy to verify. All properties of the inner product are trivial except for the positivity property. For that, we only need to show that $\|v\|_{\cH} = 0$ implies $v \equiv 0$ in $\Omega$.

Assume that $\|v\|_{\cH} = 0$.  Then $\|v^{(1)}\|_{H^1(\Omega_1)} = \|v\|_{\cH} = 0$, which implies that $v^{(1)} \equiv 0$ in $\Omega_1$. Let $v^{(2)} \coloneqq \restr{v}{\Omega_2}$. Since $v \in H^1(\Omega)$, then $v^{(2)} = v^{(1)} = 0$  on $\Sigma = \partial\Omega_1$. 
Combining with Condition $\cA$, $v^{(2)}$ satisfies \eqref{eq:v2-from-v1} with $v^{(1)} = 0$. 
The only solution to the above problem is $v^{(2)} \equiv 0$ in $\Omega_2$. 
Thus $v \equiv 0$ in $\Omega$.

Next, we show that $(\cH,\|\cdot\|_{\cH})$ is complete. 
Let $\{v_k\}_{k=1}^\infty$ be a Cauchy sequence in $\cH$.
We need to show that there exists a function $v \in \cH$ such that $\|v_k - v \|_{\cH} \to 0$ as $k \to \infty$.

For any $\varepsilon > 0$, there exists $K \in\mathbb{N}$ such that $\|v_m - v_k\|_{\cH} < \varepsilon$, for all $m,k > K$.  Let $v_\ell^{(1)} \coloneqq \restr{v_\ell}{\Omega_1}$, for any $\ell\in\mathbb{N}$.  By the definition of $\|\cdot\|_{\cH}$, we have
\begin{align*}	
	\|v_m^{(1)} - v_k^{(1)}\|_{H^1(\Omega_1)}
	= \|v_m - v_k\|_{\cH}
	< \varepsilon,
	\quad \text{for all } m,k > K.
\end{align*}
Hence $\{v_k^{(1)}\}$ is a Cauchy sequence in $H^1(\Omega_1)$. Since $H^1(\Omega_1)$ is complete, there exists a function $v^{(1)} \in H^1(\Omega_1)$ such that $v_k^{(1)} \to v^{(1)}$ strongly in $H^1(\Omega_1)$, as $k\to\infty$.
		
Now, let $v^{(2)} \in H^1(\Omega_2)$ be the unique solution to the problem \eqref{eq:v2-from-v1} with $v^{(2)} = v^{(1)}$ on $ \Sigma = \partial\Omega_1. $
Define
\begin{align*}
	v(\bfx) \coloneqq	\begin{cases}
			v^{(1)}(\bfx), & \text{if } \bfx \in  \Omega_1,\\
			v^{(2)}(\bfx), & \text{if } \bfx \in  \Omega_2.
		\end{cases}
\end{align*}
Then, $v \in \cH$ and
	$\|v_k - v\|_{\cH}
	= \|v_k^{(1)} - v^{(1)}\|_{H^1(\Omega_1)} \to 0$,
	as $k\to\infty$. 
That completes the proof. 
\end{proof}
{\begin{rema}[Equivalence of norms on $\mathcal{H}$]\label{equiv-norm-cH}
  Combining Corollary \ref{h(Omega_1)=scriptH} and Proposition \ref{H-is-Hilbert}, we have that, for any $v\in \mathcal{H}$ and $v^{(1)} = v|_{\Omega_1}$, since $\|v\|_{\mathcal{H}} = \|v^{(1)}\|_{H^{1}(\Omega_1)}$, there exists a positive constant $C$ such that 
\[\|v\|_{\mathcal{H}} \leq \|v\|_{H^{1}(\Omega)} \leq C \|v\|_{\mathcal{H}}.\]
That is, the norms $\|\cdot\|_{H^{1}(\Omega)}$,  $\|\cdot\|_{H^{1}(\Omega_1)}$, and $\|\cdot\|_{\mathcal{H}}$ are all equivalent on $\mathcal{H}$.  As a consequence, $\mathcal{H}$ is a closed subspace of $H^{1}(\Omega)$. 
\end{rema}}

\subsection{The restriction and the extension operators}

In this subsection, we study two operators: the restriction on $\Omega_1$ and the extension to $\Omega$.

Let $B: \cH \to H^1(\Omega_1)$ be the restriction operator given by $B(v) = \restr{v}{\Omega_1}$, for $v \in \cH$.  Then from the definition of the space $\cH$, $B$ is an isometric isomorphism,  that is, $B$ is an isomorphism  and $\|B(v)\|_{H^1(\Omega_1)} = \|v\|_{\cH}$, for any $v \in \cH$. Let $\cH'$ be the dual space of $\cH$. Define the adjoint operator $B^*: [H^1(\Omega_1)]' \to \cH'$ by
    \begin{align*}
B^*[f_1](v) = f_1\left(\restr{v}{\Omega_1}\right), \quad \text{for any } f_1 \in [H^1(\Omega_1)]' \text{ and any } v\in \cH.
    \end{align*}
It then follows from \cite[Theorem 4.13.4]{Friedman1982} that $B^*$ is an isometric isomorphism as well. The above relations show that Problem \hyperlink{po}{$(P_0)$} can equivalently be phrased in terms of bilinear forms over $\mathcal{H}$, though we adopt a more energetic perspective in what follows.

Next, we introduce the extension operator $\cL : H^1(\Omega_1) \to H^1(\Omega_2)$ {defined} by $\cL(v^{(1)}) = v^{(2)}$, {where for any $v^{(1)} \in H^1(\Omega_1)$, the function $v^{(2)}\in H^1(\Omega_2)$ is the unique solution to the problem \eqref{eq:v2-from-v1} guaranteed by Corollary \ref{h(Omega_1)=scriptH}. }

    \begin{proposition}\label{prop:L-bnd}
$\cL$ is a bounded linear operator.
    \end{proposition}

    \begin{proof}
For the linearity of $\cL$, take any $v^{(1)}, w^{(1)} \in H^1(\Omega_1)$ and any $a,b \in \R$. 
	By the definition of $\cL$ and the linearity of the operator $v \mapsto -\Delta v + \beta v$, 
    $a \cL(v^{(1)}) + b \cL(w^{(1)})$ is a solution to the problem
	\begin{align*}
-\Delta v + \beta v &= 0, \quad \text{in } \Omega_2, 
    \\
v &= a v^{(1)} + b w^{(1)}, \quad \text{on } \Sigma = \partial\Omega_1, 
    \\
\boldsymbol{n}_2 \cdot 	\nabla v  &= 0, \quad \text{on } \partial\Omega.
    \end{align*}
Since the above problem has a unique solution, which is defined by $\cL(a v^{(1)} + b w^{(1)})$, 
	then $a \cL(v^{(1)}) + b \cL(w^{(1)}) = \cL(a v^{(1)} + b w^{(1)})$.
	
The boundedness of $\cL$ follows from Corollary \ref{h(Omega_1)=scriptH}.  
    \end{proof}

    \section{Energy Setting for the Problem}
    \label{sec:energy}
{In this section, we will demonstrate the convergence of a sequence of minimizers of the parametrized energy functionals $\cE_\alpha$ defined in \eqref{E_alpha-energy}, as $\alpha\searrow 0$. We do that by establishing 
variational convergence of  $\cE_\alpha$ restricted to an appropriate subspace.} However, we need to carefully define the limiting energy and its set of admissible functions. We have already seen, and, in fact, it is quite obvious that the solution $u_\alpha$ of Problem ~\eqref{eqn:alpha-prob} is connected to a minimization problem. {The choice of an appropriate set of admissible functions is motivated by the following result. }

    \begin{proposition}\label{prop:min-Ealpha}
For each $\alpha>0$, the energy functional $\cE_\alpha$ has a unique minimizer $u_\alpha \in H^1(\Omega)$ and it satisfies Problem~\eqref{eqn:alpha-prob}. Moreover, $u_\alpha \in \cH$ and  
    \begin{align*}
\min_{u\in H^1(\Omega)} \cE_\alpha[u]	=  \min_{u\in\cH} \cE_\alpha[u].
    \end{align*}
    \end{proposition}
{\begin{proof}
    The fact that the quadratic energy $\cE_\alpha$ has a unique minimizer $u_\alpha \in H^1(\Omega)$ follows from the Lax-Milgram theorem. Moreover, the minimizer satisfies Problem~\eqref{eqn:alpha-prob} in the weak sense. Writing 
   \begin{align*}
u_\alpha(\bfx) = 
    \begin{cases}
u_\alpha^{(1)}(\bfx),  &\text{if } \bfx \in  \Omega_1, 
    \\
u_\alpha^{(2)}(\bfx),  &\text{if } \bfx \in  \Omega_2,
	\end{cases}
    \end{align*}
we see from equation \eqref{eqn:alpha-prob} that $u_\alpha$ satisfies Condition $\mathcal{A}$. That is, the minimizer $u_\alpha$ belongs to $\mathcal{H}$. Therefore, since $\mathcal{H}\subseteq H^{1}(\Omega)$, 
\[
\cE_\alpha[u_\alpha] = \min_{u\in H^1(\Omega)} \cE_\alpha[u]	\leq \min_{u\in\cH} \cE_\alpha[u] \leq \cE_\alpha[u_\alpha], 
\]
where in the last inequality we used $u_\alpha\in \mathcal{H}$. 
\end{proof} 
}

{Since our goal is to establish the convergence of minimizers, the above proposition shows that it suffices to study the variational convergence of the parameterized energy functionals $\cE_\alpha$ restricted to the closed subspace $\mathcal{H}$. Moreover, it allows us to identify the minimization problem associated with $u_0$, the solution to Problem~\hyperlink{po}{$(P_0)$}. To that end,  }
define the energy functional $\cE_0 : \cH  \to \R$ by
\begin{align*}
	\cE_0[u] =
	\int_{\Omega_1}
	\frac{1}{2} \left( |\nabla u|^2 + \gamma u^2 \right)\mathrm{d}\bfx 	+
	\int_{\Sigma}\frac{1}{2} \kappa u^2  \mathrm{d}S- \langle q, u\rangle_{\Omega_1}+ \langle g, u\rangle_{\Sigma},
	\quad u \in \cH, 
\end{align*}
Note carefully that the volume integral is over $\Omega_1$ only.

    \begin{proposition}\label{prop:min-E0}
The energy functional $\cE_0$ has a unique minimizer $u$ in $\cH$, and, moreover, the minimizer solves Problem~\hyperlink{po}{$(P_0)$}.  Conversely, if $u$ is a solution to Problem~\hyperlink{po}{$(P_0)$}, then $u$ is a minimizer of $\cE_0$ in $\cH$.
    \end{proposition}

\begin{proof}
Define the energy functional $\cE_0^1 : H^1(\Omega_1)  \to \R$ via
    \begin{align*}
\cE_0^1[w] \coloneqq \int_{\Omega_1} \frac{1}{2} \left( | \nabla w|^2 + \gamma w^2 \right)\mathrm{d}x + \int_{\Sigma} \frac{1}{2} \kappa w^2  \mathrm{d}S- \langle q,  w\rangle_{\Omega_1}	+\langle g, w\rangle_{\Sigma}, \ \forall \,  w \in H^1(\Omega_1).
    \end{align*}
One can verify that $\cE_0^1$ is a strictly convex and coercive functional {(see the proof of Theorem \ref{thm:compactness})}. Moreover, by continuity of the trace operator {$\Tr_{\Sigma}$}, $\cE_0^1$ is weakly lower semicontinuous in $H^{1}(\Omega_1)$. Thus, by the direct method of the calculus of variations, $\cE_0^1$ has a unique minimizer in $H^{1}(\Omega_1)$. 
	
Note also that by definition of $\cH$, $\cE_0$ is a strictly convex, coercive and weakly lower semicontinuous functional on $\cH$ (cf., Theorem \ref{thm:compactness}) and so a unique minimizer exists. Moreover, 
   {\[
\min_{u \in \cH} \cE_0[u] = \min_{u\in \cH}\cE_0^1[u|_{\Omega_{1}}]  = \min_{w\in H^1(\Omega_1)}\cE_0^1[w].  
    \]}
{Now if $w = \operatorname{argmin}_{u\in H^1(\Omega_1)}\cE_0^1[u]$, we define $u^{(1)} \coloneqq w \in H^1(\Omega_1)$. } The unique function $u\in \cH$ corresponding to $u^{(1)}$ (by Corollary \ref{h(Omega_1)=scriptH}) minimizes $\cE_0[u]$ over $\cH$. By definition of the extension, $u$ solves Problem~{\hyperlink{po}{$(P_0)$}}, which is the Euler--Lagrange equation of the variational problem.

Conversely, suppose that $u \in H^1(\Omega)$ is a solution to Problem~{\hyperlink{po}{$(P_0)$}},
with $u^{(1)} \coloneqq \restr{u}{\Omega_1}$ 
and $u^{(2)} \coloneqq \restr{u}{\Omega_2}$. 
Then $u$ satisfies Condition $\cA$, hence $u \in \cH$. 
The boundary value problem satisfied by $u^{(1)} \in H^1(\Omega_1)$ is the Euler-Lagrange equation corresponding to 
minimizing $\cE_0^1$ over $H^1(\Omega_1)$. Thus, $u^{(1)}$ is the unique minimizer of $\cE_0^1$. 
Moreover, $\cE_0[u] = \cE_0^1[u^{(1)}]$.
For any $v \in \cH$, define $v^{(1)} \coloneqq \restr{v}{\Omega_1} \in H^1(\Omega_1)$ and $v^{(2)} \coloneqq \restr{v}{\Omega_2} \in H^1(\Omega_2)$, 
then $\cE_0^1[v^{(1)}] = \cE_0[v]$. 
Since $u^{(1)}$ is a minimizer of $\cE_0^1$ in $H^1(\Omega_1)$, then $\cE_0^1[u^{(1)}] \leq \cE_0^1[v^{(1)}]$, 
which implies that $\cE_0[u] \leq \cE_0[v]$. 
Hence, $u$ is a minimizer of $\cE_0$ over $\cH$.
\end{proof}

    \section{$\Gamma$--Convergence of Energy Functionals}
    \label{sec:Gamma-conv}

Define the energy functionals
\begin{align*}
	\cF_\alpha[u] \coloneqq
	\begin{cases}
		\cE_\alpha[u], & u\in\cH,\\
		\infty,     & u\in L^2(\Omega)\setminus \cH,
	\end{cases}
	\quad \text{and} \quad 
	\cF_0[u] \coloneqq
	\begin{cases}
		\cE_0[u], & u\in\cH,\\
		\infty, & u\in L^2(\Omega)\setminus \cH.
	\end{cases}
\end{align*}
In this section, we will show that $\cF_\alpha$ $\Gamma$-converges to $\cF_0$ in the strong $L^2(\Omega)$ topology, as $\alpha \searrow 0$. This analysis requires three basic steps: (i) a compactness result, which helps to define the proper topology, (ii) a $\liminf$ result, and (iii) a $\limsup$ result which involves the identification of a recovery sequence. See~\cite{Braides2002} for the basic background.

\subsection{Compactness}

We begin with the following estimate on sequences based on the the estimate of the sequence of energies.

{ \begin{proposition} \label{prop:normest-from-energy}
    There exists positive constants $C_1$ and $C_2$ such that for any sequence $\{u_k\}\subset L^{2}(\Omega)$ and for any sequence $\{\alpha_k\}\subset(0, 1)$ we have 
    \[
    \|u^{(1)}_{k}\|^{2}_{H^{1}(\Omega_1)} \leq C_1\left( M + C_2\right)
    \]
  where  $u_k^{(1)} = \restr{u_k}{\Omega_1}$ and $M = \sup_{k}\mathcal{F}_{\alpha_k}[u_k]$. If $M<\infty$, then $u_k\in \mathcal{H}$ for all $k$ and for some constant $c>0$ (independent of $k$)and 
  \[
 c \|u_k\|^2_{H^{1}(\Omega)} \leq \|u_k\|^2_{\mathcal{H}} = \|u^{(1)}_{k}\|^{2}_{H^{1}(\Omega_1)} \leq C_1\left( M + C_2\right)
  \]
\end{proposition}
\begin{proof}
It suffices to prove the inequality in the case that $M = \sup_{k}\mathcal{F}_{\alpha_k}[u_k] < \infty$. 
    For each $k = 1,2,\ldots$, since $\cF_{\alpha_k}[u_k] \leq  M < \infty$, 
by the definition of $\cF_{\alpha_k}$, 
we have $u_k \in \cH$ and 
    \begin{align*}
\cF_{\alpha_k}[u_k] =& \int_{\Omega_1} \frac{1}{2} \left( \left| \nabla u_k^{(1)} \right|^2 + \gamma \left| u_k^{(1)} \right|^2 \right)\mathrm{d}\bfx	+ \int_{\Sigma} \frac{1}{2} \kappa \left| u_k^{(1)} \right|^2 \mathrm{d}S- \langle q, u_k^{(1)} \rangle_{\Omega_1} 
    \\
& + \langle g, u_k^{(1)} \rangle_{\Sigma} + \int_{\Omega_2} \frac{1}{2} \alpha_k \left( \left| \nabla u_k^{(2)} \right|^2 + \beta \left| u_k^{(2)} \right|^2 \right) \, \mathrm{d}\bfx,
\end{align*}
where $u_k^{(1)} = \restr{u_k}{\Omega_1}$ and $u_k^{(2)} = \restr{u_k}{\Omega_2}$. 
Let $\omega \coloneqq \min \{ \gamma, 1 \} > 0$, then
\begin{align}\label{est-0}
	\int_{\Omega_1} 
	\left[ \frac{1}{2} \left( \left| \nabla u_k^{(1)} \right|^2 + \gamma \left| u_k^{(1)} \right|^2 \right) \right] \mathrm{d}\bfx \geq \frac{\omega}{2} \| u_k^{(1)} \|^2_{H^1(\Omega_1)}.
\end{align}
For any constant $b > 0$, using Young's inequality, we get
    \begin{equation}\label{est-1}
|\langle q, u_k^{(1)} \rangle| \leq \|q\|_{[H^{1}(\Omega_1)]'}\|u_k^{(1)}\|_{H^{1}(\Omega_1)} \leq b \| u_k^{(1)} \|_{H^1(\Omega_1)}^2 + \frac{1}{4b} \| q\|_{[H^{1}(\Omega_1)]'}^2. 
    \end{equation}
 
Since the trace operator $\Tr_{\Sigma} : H^1(\Omega_1) \to H^{\half}(\Sigma)$ is continuous, there exists $C_3>0$ {such that, for any $w \in H^1(\Omega_1)$,} we have
\begin{align}\label{trace-ineq}
	\|\Tr_{\Sigma}(w)\|_{H^{\half}(\Sigma)} \leq C_3 \| w \|_{H^1(\Omega_1)}.
\end{align}
Then, by Young's inequality, we have, for any $k = 1,2,\ldots$,
\begin{equation}\label{est-2}
	\begin{split}
		| \langle g, u_k^{(1)} \rangle | 
        &\leq \|g\|_{H^{-\half}(\Sigma)}\|\Tr_{\Sigma}(u_k^{(1)})\|_{H^{\half}(\Sigma)} \\ &\leq C_3 \|g\|_{H^{-\half}(\Sigma)}\|u_k^{(1)}\|_{H^{1}(\Omega_1)} 
    \\
&\leq C_3^2  b \| u_k^{(1)} \|_{H^1(\Omega_1)}^2 + \frac{1}{4b}\|g\|_{H^{-\half}(\Sigma)}^2 ,
	\end{split}
\end{equation} 
for all $b>0$.  Combining \eqref{est-0}, \eqref{est-1} and \eqref{est-2}, we get
\begin{align*}
	M \geq & \cF_{\alpha_k}[u_k] 
	\geq \left( \frac{\omega}{2} - (C_3^2 + 1)b  \right) \| u_k^{(1)} \|^2_{H^1(\Omega_1)} - \frac{M_1}{4b} 
    + \frac{1}{2} \kappa \| \Tr_{\Sigma}(u_k^{(1)}) \|_{L^2(\Sigma)}^2 
\end{align*}
for all $k = 1,2,\ldots$, where $M_1 \coloneqq \|q\|_{[H^{1}(\Omega_1)]'}^2 + \|g\|_{H^{-\half}(\Sigma)}^2$. Choose $b = \omega / (4C_3^2 + 4)$, then for all $k = 1,2,\ldots$,
    \begin{align}\label{est-4}
\| u_k^{(1)} \|_{H^1(\Omega_1)}^2 \leq \frac{4}{\omega} \left( M + \frac{(C_3^2 + 1) M_1}{\omega} \right). 
    \end{align}
The proposition will be proved if we set $C_1 = \frac{4}{\omega}$ and $C_2 = \frac{(C_3^2 + 1) M_1}{\omega}.$ 
The bounds of the sequence on $H^{1}(\Omega)$ follows from  the equivalence of norms stated in Remark \eqref{equiv-norm-cH}. 

\end{proof}
}

    \begin{theorem}[Compactness]
    \label{thm:compactness}
Let $\{ \alpha_k \} \subset (0,1)$ be a sequence of numbers such that $\alpha_k \searrow 0$ as $k \to \infty$. 
	Let $\{ u_k \} \subset L^2(\Omega)$ be a sequence of functions such that, for any $k = 1,2,\ldots$, 
	$\cF_{\alpha_k}[u_k] < M < \infty$, for some $M > 0$ independent of $k$. 
	Then, there exists a subsequence $\{ u_{k_j} \} \subseteq \{ u_k \}$ and a function $u \in \cH$ such that, as $j \to \infty$,
    \begin{align*}
u_{k_j} &\to u \text{ strongly in $L^2(\Omega)$,} {\qquad \Tr_{\Sigma}(u_{k_j}) \to \Tr_{\Sigma}(u)  \text{ strongly in $L^2(\Sigma)$},}
    \\
u_{k_j} &\rightharpoonup u \text{ weakly in $H^1(\Omega)$,} {\qquad \Tr_{\Sigma}(u_{k_j}) \rightharpoonup \Tr_{\Sigma} (u)  \text{ weakly in $H^{\half}(\Sigma)$, and }} 
    \\
u_{k_j} &\to u \text{ a.e. in } \Omega. 
	\end{align*}

\end{theorem}	

\begin{proof}	
By the definition of $\cF_{\alpha_k}$, and assumption,  
we have $u_k \in \cH$ for each $k\in \mathbb{N}$. 
We may also apply Proposition \ref{prop:normest-from-energy} to conclude that $\{u^{(1)}_k\}$  is uniformly bounded in $H^{1}(\Omega_1)$.  
It then follows by \eqref{trace-ineq} that   $\| \Tr_{\Sigma}(u_k^{(1)}) \|_{H^{\half}(\Sigma)}$ is uniformly bounded (and therefore precompact in $L^{2}(\Sigma)$).  As a consequence,  there exists a subsequence $\{ u_{k_j}^{(1)} \} \subseteq \{ u_k^{(1)} \}$ and a function $u^{(1)} \in H^1(\Omega_1)$ such that as $j \to \infty$	
    \begin{align*}
u_{k_j}^{(1)} &\to u^{(1)} \text{ strongly in } L^2(\Omega_1),
    \\
\Tr_{\Sigma}(u_{k_j}^{(1)}) &\to \Tr_{\Sigma}(u^{(1)}) \text{ strongly in } L^2(\Sigma), 
    \\
u_{k_j}^{(1)} &\rightharpoonup u^{(1)} \text{ weakly in } H^1(\Omega_1),  
    \\
\Tr_{\Sigma}(u_{k_j}^{(1)}) &\rightharpoonup \Tr_{\Sigma}(u^{(1)}) \text{ weakly in } H^{\half}(\Sigma), 
    \\
u_{k_j}^{(1)} &\to u^{(1)} \text{ a.e. in } \Omega_1. 
    \end{align*}
Since $u_k \in \cH$, by the definition of the space $\cH$, $u_{k_j}^{(2)} = \cL(u_{k_j}^{(1)})$. 
Since $\cL: H^1(\Omega_1) \to H^1(\Omega_2)$ is a bounded linear operator, using \eqref{est-4}, 
 $u_{k_j}^{(2)}$ is uniformly bounded in $H^1(\Omega_2)$. 
Hence, there exists a further subsequence of $\{ u_{k_j}^{(2)} \}$ (not relabeled) and a function $u^{(2)} \in H^1(\Omega_2)$  such that
    \begin{align*}
u_{k_j}^{(2)} &\to u^{(2)} \text{ strongly in } L^2(\Omega_2), 
    \\
u_{k_j}^{(2)} &\rightharpoonup u^{(2)} \text{ weakly in } H^1(\Omega_2),  
    \\
\Tr_{\Sigma}(u_{k_j}^{(2)}) &\rightharpoonup \Tr_{\Sigma}(u^{(2)}) \text{ weakly in } H^{\half}(\Sigma),  
    \\
u_{k_j}^{(2)} &\to u^{(2)} \text{ a.e. in } \Omega_2,
    \end{align*}
as $j \to \infty$.		
Since $u_{k_j}^{(1)} \rightharpoonup u^{(1)}$ weakly in $H^1(\Omega_1)$,
by continuity of the trace operator $\Tr_{\Sigma}: H^1(\Omega_1) \to H^{\half}(\Sigma)$, 
we have $\Tr_{\Sigma}(u_{k_j}^{(1)}) \rightharpoonup \Tr_{\Sigma}(u^{(1)})$ weakly in $H^{\half}(\Sigma)$.
Passing to the limit in the weak formulation of the mixed problem in $\Omega_2$ and using uniqueness of weak solutions yields $u^{(2)}=\cL(u^{(1)})$ with
\begin{align*}
	u(\bfx) = \begin{cases}
		u^{(1)}(\bfx),  &\text{if } \bfx \in  \Omega_1, \\
		u^{(2)}(\bfx),  &\text{if } \bfx \in  \Omega_2,
	\end{cases}
\end{align*}
belonging to $\cH$. Together, $u$ satisfies all the convergences asserted in the theorem. 
\end{proof}

\subsection{Liminf Inequality} \mbox{ }

    \begin{theorem}[Liminf Inequality]\label{thm:liminf-ineq}
Let $\{ \alpha_k \} \subset (0,1)$ be a sequence of numbers such that $\alpha_k \searrow 0$ as $k \to \infty$.  For any function $u \in L^2(\Omega)$ and any sequence $\{ u_k \} \subset L^2(\Omega)$  that satisfies $u_k \to u$ strongly in $L^2(\Omega)$ as $k \to \infty$, we have
    \begin{align}\label{liminf}
\liminf_{k \to \infty} \cF_{\alpha_k}[u_k] \geq \cF_0[u].
	\end{align}
    \end{theorem}

    \begin{proof}	
If $\liminf_{k \to \infty} \cF_{\alpha_k}[u_k] = \infty$, then \eqref{liminf} is trivial.  Therefore, we only need to consider the case for which $\liminf_{k \to \infty} \cF_{\alpha_k}[u_k] = M_0 < \infty$.  We may choose a subsequence (not relabeled) such that $u_{k} \to u$ a.e. in $\Omega$ and $\cF_{\alpha_{k}}[u_{k}] \to M_0$ as $k \to \infty$. Consequently, $\{\cF_{\alpha_{k}}[u_{k}]\}$ is bounded, and thus $u_{k} \in \cH$, for all $k$. Then by the compactness result in Theorem~\ref{thm:compactness}, there exists a further subsequence of $\{ u_{k} \}$ (not relabeled) and a function $v \in \cH$ such that $u_{k} \rightharpoonup  v$ weakly in  $H^1(\Omega)$ as $k \to \infty$. {The convergence is strong in $L^2(\Omega)$ and the traces converge strongly in  $L^2(\Sigma)$.
	 
 Since $u_{k} \to u$ a.e. in $\Omega$ as $k \to \infty$, we have  $u = v$ a.e. in $\Omega$, which implies that $u \in \cH$ and $u_{k}$ converges to $u$ weakly in $H^1(\Omega)$, and strongly in $L^{2}(\Omega)$, while its traces converge weakly in $H^{\half}(\Sigma)$, strongly in $L^2(\Sigma)$ as $k \to \infty$. }
 
Write $u_k^{(1)} = \restr{u_k}{\Omega_1}$ and $u_k^{(2)} = \restr{u_k}{\Omega_2}$. It then follows from the weak convergence, the lower semicontinuity of norms and the fact that $q\in [H^{1}(\Omega_1)]'$ that 
    \[
\liminf_{k \to \infty} \int_{\Omega_1}  \left| \nabla u_k^{(1)} \right|^2 \mathrm{d}\bfx \geq \int_{\Omega_1} \left| \nabla u^{(1)} \right|^2  \mathrm{d}\bfx
    \]
and
    \[
\lim_{k\to \infty} \langle q, u^{(1)}_k\rangle_{\Omega_1} = \langle q, u^{(1)}\rangle_{\Omega_1}. 
	\]
Moreover, the weak convergence of traces in $H^{\half}(\Sigma)$ and the strong convergence in $L^{2}(\Omega)$ and $L^2(\Sigma)$ imply that 
    \begin{align*}
\lim_{k \to \infty} \int_{\Omega_1} \left| u_k^{(1)} \right|^2 \mathrm{d}\bfx&=\int_{\Omega_1} \left| u^{(1)} \right|^2  \mathrm{d}\bfx, 
    \\
\lim_{k \to \infty} \int_{\Sigma}  \left| u_k^{(1)} \right|^2 \mathrm{d}S &= \int_{\Sigma} \left| u^{(1)} \right|^2 \mathrm{d}S,
    \\
\lim_{k \to \infty} \langle g,  u_k^{(1)}\rangle_{\Sigma}  &= \langle g,  u^{(1)}\rangle_{\Sigma}.
	\end{align*}
Combining all of the above, along with 
	\begin{align*}
\int_{\Omega_2} \frac{1}{2} \alpha_k \left( \left| \nabla u_k^{(2)} \right|^2 + \beta \left| u_k^{(2)} \right|^2 \right)  \mathrm{d}\bfx\geq 0,\quad \text{for all } k\geq 1
	\end{align*}
we obtain $\liminf_{k \to \infty} \cF_{\alpha_k}[u_k] \geq \cF_0[u]$.
    \end{proof}

\subsection{Limsup Inequality and Recovery Sequence} \mbox{ }

\begin{theorem}[Limsup Inequality]\label{thm:limsup-ineq}
	Let $\{ \alpha_k \} \subset (0,1)$ be a sequence of numbers such that $\alpha_k \searrow 0$ as $k \to \infty$. 
	For any function $u \in L^2(\Omega)$, there exists a sequence $\{ u_k \} \subset L^2(\Omega)$ such that $u_k \to u$ strongly in $L^2(\Omega)$ as $k \to \infty$, and
	\begin{align*}
		\limsup_{k \to \infty} \cF_{\alpha_k}[u_k] \leq \cF_0[u].
	\end{align*}
\end{theorem}

\begin{proof}	
	In the case $u \in L^2(\Omega)$ with $\cF_0[u] = \infty$, any sequence $\{ u_k \} \subset L^2(\Omega)$ that converges strongly to $u$ in $L^2(\Omega)$ can serve as a recovery sequence. 
	Now, assume that $u \in L^2(\Omega)$ and $\cF_0[u] < \infty$, 
	then $u \in \cH$ and 
	\begin{align*}
		\cF_0[u] = \int_{\Omega_1}  \frac{1}{2} \left(\left| \nabla u^{(1)} \right|^2 + \gamma \left| u^{(1)} \right|^2 \right)\mathrm{d}\bfx + \int_{\Sigma} \frac{1}{2} \kappa \left| u^{(1)} \right|^2 \mathrm{d}S - \langle q, u^{(1)}\rangle_{\Omega_1} + \langle g, u^{(1)}  \rangle_{\Sigma},
	\end{align*}
	where $u^{(1)} = \restr{u}{\Omega_1}$. Note also that $u^{(2)} = \restr{u}{\Omega_2}\in H^{1}(\Omega_2)$. 
    
    Now  choose $u_k = u$, for all $k = 1,2,\ldots$, then $u_k \in \cH$, which implies that, for all $k \in\mathbb{N}$, 
	\begin{align*}
\cF_{\alpha_k}[u_k] =& \int_{\Omega_1}  \frac{1}{2} \left( \left| \nabla u^{(1)} \right|^2 + \gamma \left| u^{(1)} \right|^2 \right)  \mathrm{d}\bfx	+ \int_{\Sigma}\frac{1}{2} \kappa \left| u^{(1)} \right|^2\mathrm{d}S  - \langle q, u^{(1)}\rangle_{\Omega_1} 
    \\
& + \langle g, u^{(1)}  \rangle_{\Sigma} + \int_{\Omega_2} \frac{1}{2} \alpha_k \left( \left| \nabla u^{(2)} \right|^2 + \beta \left| u^{(2)} \right|^2 \right)   \mathrm{d}\bfx. 
	\end{align*}
Since $u^{(2)} \in H^1(\Omega_2)$ and $\alpha_k \searrow 0$, then the last term goes to zero as $k \to \infty$. Therefore, 
	\begin{align*}
\limsup_{k \to \infty} \cF_{\alpha_k}[u_k] =\lim_{k \to \infty} \cF_{\alpha_k}[u_k] = \cF_0[u], 
	\end{align*}
which implies that the constant sequence serves as a recovery sequence for $u$.
    \end{proof}

  {\subsection{$\Gamma$--Convergence and Convergence of Minimizers}
We finally state and prove the convergence of minimizers. 
\begin{proposition}\label{convergence-of-minimizers}
The sequence $\cF_\alpha$ $\Gamma$-converges to $\cF_0$ in the strong $L^2(\Omega)$ topology, as $\alpha \to 0$.  
Moreover, the sequence of minimizers $\{u_\alpha= \arg \min_{u\in \mathcal{H}}\cF_{\alpha}[u]\}$ is convergent in $L^{2}(\Omega)$ and as $\alpha\to 0$
\[
   u_{\alpha} \to u_{0} = \arg \min_{u\in \mathcal{H}} \cF_{0}[u] \quad \text{strongly in $L^{2}(\Omega)$.}
 \]   
\end{proposition}
\begin{proof}
The $\Gamma$-convergence of the energies follows from the liminf and limsup inequalities proved earlier and using the general theory of $\Gamma$--convergence, which is reviewed in~\cite{Braides2002}. 
 Notice that by Theorem \ref{thm:compactness}, the sequence $\cF_\alpha$ is equicoercive.
As a consequence of the $\Gamma$--convergence of the parametrized energy functionals, the Fundamental Theorem of $\Gamma$--convergence (see, e.g., \cite{Braides2014}, Theorem~2.1) implies that the sequence of minimizers $u_\alpha$ of $\cF_{\alpha}$, whose existence and uniqueness are guaranteed by Proposition \ref{prop:min-Ealpha}, converges strongly in $L^2(\Omega)$ to the minimizer $u_0$ of $\cF_{0}$, whose existence and uniqueness are guaranteed by Proposition \ref{prop:min-E0}.

\end{proof}

We close this section by noting that the minimizers are uniformly bounded in $H^{1}(\Omega)$, that is,  there is a universal constant $C>0$ such that for all $\alpha\in (0, 1)$
 \begin{equation}\label{universal-bound-minimizers}
\| u_{\alpha}\|_{\cH} \leq \| u_\alpha\|_{H^{1}(\Omega)} \leq C. 
 \end{equation}
 Indeed, inequality\eqref{universal-bound-minimizers} follows from Proposition \ref{prop:normest-from-energy} after noticing that by minimality, for all $\alpha\in (0, 1)$, 
\[
\cF_{\alpha}[u_\alpha] \leq \cF_\alpha[0] = 0. 
\] 
Moreover, thanks to the Euler-Lagrange equations, convergence can also be established using the stronger $H^{1}$-norm, and that is the goal of the next section.
}
    \section{Strong $H^1(\Omega)-$Convergence of the Solutions $u_\alpha$}
    \label{sec:H1-conv} 

In this section, we prove the convergence of $u_\alpha$ to $u_0$ in $H^{1}(\Omega)$, and we also establish the rate of convergence, which {is bounded above by a constant times} $\alpha$, as $\alpha\searrow 0$. We first state and prove a technical lemma that is related to the well-known Normal Trace Theorem, which guarantees that, if $u\in H^{1}(U)$ and $\Delta u\in L^{2}(U)$, then $\boldsymbol{n} \cdot \nabla u \in H^{-\half}(\partial U)$ \cite{ern2021finite}.

Let $\cW$ be the set of all functions $v \in H^1(\Omega_2)$ satisfying
\begin{equation}\label{W-equation}
\begin{aligned}
    -\Delta v + \beta v &= 0, \quad \text{in } \Omega_2, \\
	 \boldsymbol{n}_2 \cdot \nabla v  &= 0, \quad \text{on } \partial\Omega,
\end{aligned}
\end{equation}
in the weak sense, that is,
\begin{align*}
    \int_{\Omega_2} 
    (\nabla v \cdot \nabla\xi + \beta v \xi) \, \mathrm{d}\bfx    
    = 0,
\end{align*}
for any $\xi \in H^1(\Omega_2)$ such that $\Tr_{\Sigma} (\xi) = 0$ on $\Sigma$. {By Proposition~\ref{prop:Helmholtz}, for a given $g\in H^{\half}(\Sigma)$, there exists a unique function $v_g \in H^{1}(\Omega_2)$ such that $\Tr_{\Sigma}(v_g) = g$ on $\Sigma$ and $v_g$ satisfies \eqref{W-equation} in the weak sense. Thus, $\cW$ has nontrivial elements. }

    \begin{lemma}
    \label{lem:tech-1}
 There exists a constant $C > 0$, depending only on $\beta$ and $\Omega_2$,  such that for any $v\in \cW$ and $\psi \in H^{\half}(\Sigma)$ 
    \begin{align}\label{N-bound}
\left| \langle \boldsymbol{n}_1 \cdot \nabla v ,\psi \rangle_{\Sigma}  \right| \leq C \|v\|_{H^1(\Omega_2)} \|\psi\|_{H^{\half}(\Sigma)}. 
    \end{align}
    \end{lemma}
    
    \begin{proof}
{Let $v \in \cW$. Then by definition \eqref{W-equation}, since $-\Delta v + \beta v = 0$ weakly, then its distributional Laplacian $\Delta v $ satisfies  \[\operatorname{div}(\nabla v) = \Delta v = \beta v \in L^2(\Omega_2).
\]
Hence, $\nabla v \in H(\operatorname{div}; \Omega_2)$, the set of functions in $L^{2}(\Omega_2)$ whose divergence is also in $L^2(\Omega_2)$. By the Normal Trace Theorem \cite{ern2021finite}, the normal derivative $\boldsymbol{n}_1 \cdot \nabla v\big\vert{}_\Sigma \in H^{-\half}(\Sigma)$ is well defined and is given, for $\psi \in H^{\half}(\Sigma)$, by the generalized Green's formula 
\[\langle \boldsymbol{n}_1 \cdot \nabla v\big\vert{}_\Sigma, \psi \rangle_{\Sigma} 
:=-\int_{\Omega_2} (\nabla v \cdot \nabla \Psi + \beta v \Psi) \,\mathrm{d}x
\]
where $\Psi \in H^1(\Omega_2)$ is any function satisfying $\Psi\vert{}_\Sigma = \psi$ and $\Psi\vert{}_{\partial\Omega} = 0$.  
By a standard extension theorem, we can choose $\Psi \in H^1(\Omega_2)$ satisfying $\Psi\vert{}_\Sigma = \psi$ and $\Psi\vert{}_{\partial\Omega} = 0$ such that 
	\begin{align*}
		\|\Psi\|_{H^1(\Omega_2)} \leq C \|\psi\|_{H^{\half}(\Sigma)},
	\end{align*}
for some constant $C > 0$ depending only on $\Omega_2$. Let $\Psi_1, \Psi_2 \in H^1(\Omega_2)$ be two extensions of $\psi$, both satisfying $\Psi_1\vert{}_\Sigma = \Psi_2\vert{}_\Sigma = \psi$ and $\Psi_1\vert{}_{\partial\Omega} = \Psi_2\vert{}_{\partial\Omega} = 0$. Their difference $\xi := \Psi_1 - \Psi_2$ belongs to $H^1(\Omega_2)$ and satisfies: $\xi\vert{}_\Sigma = 0$ and $ \xi\vert{}_{\partial\Omega} = 0$. Therefore, $\xi \in H^1_0(\Omega_2)$. 
Since $\xi$ is an admissible test function for \eqref{W-equation}, we have $$\int_{\Omega_2} (\nabla v \cdot \nabla \xi + \beta v \xi) \, \mathrm{d}x = 0.$$
By linearity:
\[\int_{\Omega_2} (\nabla v \cdot \nabla \Psi_1 + \beta v\Psi_1) \, \mathrm{d}x = \int_{\Omega_2} (\nabla v \cdot \nabla \Psi_2 + \beta v\Psi_2) \, \mathrm{d}x.
\]
Thus, the value depends strictly on the boundary data $\psi$, not on how $\Psi$ is chosen inside $\Omega_2$.}

Applying the Cauchy–Schwarz inequality, we have
	\begin{align*}
		\left| \langle \boldsymbol{n}_1 \cdot \nabla v ,\psi \rangle_{\Sigma}  \right|
		&\leq \left| \int_{\Omega_2} \nabla v \cdot \nabla\Psi \, \mathrm{d}\bfx \right|
        + \beta \left| \int_{\Omega_2} v \Psi \, \mathrm{d}\bfx \right| \\
		&\leq \|\nabla v\|_{L^2(\Omega_2)} \|\nabla\Psi\|_{L^2(\Omega_2)} 
        + \beta \|v\|_{L^2(\Omega_2)} \|\Psi\|_{L^2(\Omega_2)} \\
		&\leq C  \|v\|_{H^1(\Omega_2)} \|\psi\|_{H^{\half}(\Sigma)}.
	\end{align*}
The lemma has been established.
    \end{proof}

\begin{theorem}\label{thm:H1-conv}
	Let $u_\alpha$ be the solution to \eqref{eqn:alpha-prob} and $u_0$ be the solution to Problem {\hyperlink{po}{$(P_0)$}}. 
	Then, the sequence of minimizers $\{u_\alpha\}$ converges to $u_0$ strongly in $H^1(\Omega)$, as $\alpha \searrow 0$. 
	Moreover, the error is $O(\alpha)$,  
	that is, there exists a constant $C > 0$ independent of $\alpha$ such that
	\begin{align*}
		\| u_\alpha - u_0 \|_{H^1(\Omega)} \leq C\, \alpha,
	\end{align*}
	for any $\alpha \in (0,1)$. 
\end{theorem}

\begin{proof} 
Since $\cF_{\alpha}$ $\Gamma-$converges to $\cF_0$ as $\alpha\to 0$, the Fundamental Theorem of $\Gamma-$convergence  (\cite{Braides2014}, Theorem~2.1) and Proposition \eqref{convergence-of-minimizers} yields
    \begin{align*}
u_{\alpha} \to u_0 \text{ strongly in }L^2(\Omega), \text{ and } \cF_{\alpha}[u_{\alpha}] \to \cF_0[u_0], 
	\quad \text{as } \alpha \to 0.
    \end{align*}	
Hence, $\{ \cF_{\alpha}[u_{\alpha}] \}$ is a bounded sequence of finite numbers.  Setting  $u_{\alpha}^{(1)} \coloneqq \restr{u_{\alpha}}{\Omega_1}$ and $u_{\alpha}^{(2)} \coloneqq \restr{u_{\alpha}}{\Omega_2}$. Noting that  $u_{\alpha}^{(2)} = \cL(u_{\alpha}^{(1)})$, we have that $u_{\alpha}^{(2)}\in \cW$, where $\cW$ is as in Lemma~\ref{lem:tech-1}. 
Applying Proposition~\ref{prop:normest-from-energy}  we also have that $u_{\alpha}$ is a bounded sequence in $H^{1}(\Omega)$ (see inequality \eqref{universal-bound-minimizers}).  
In particular, $u^{(2)}_{\alpha}$ is a bounded sequence in $H^{1}(\Omega_2)$.

Define the errors
\begin{align*}
	e_{\alpha} &\coloneqq u_{\alpha} - u_0 \in H^1(\Omega), \\
	e_{\alpha}^{(1)} &\coloneqq \restr{e_{\alpha}}{\Omega_1} = u_{\alpha}^{(1)} - u_0^{(1)} \in H^1(\Omega_1), \\
	e_{\alpha}^{(2)} &\coloneqq \restr{e_{\alpha}}{\Omega_2} = u_{\alpha}^{(2)} - u_0^{(2)} \in H^1(\Omega_2). 
\end{align*}
These sequences of functions are all bounded in their respective domains, and the sequence $\{e_{\alpha}^{(1)}\}$ satisfies
    \begin{align*}
-\Delta e_{\alpha}^{(1)} + \gamma e_{\alpha}^{(1)} &= 0, \quad \text{in } \Omega_1, 
    \\
- \boldsymbol{n}_1 \cdot \nabla e_{\alpha}^{(1)}  &= \kappa \Tr_{\Sigma}(e_{\alpha}^{(1)})   - \alpha  \boldsymbol{n}_1 \cdot \nabla u_{\alpha}^{(2)} , 
	\quad \text{on } \Sigma = \partial\Omega_1.
    \end{align*}
Hence, $e_{\alpha}^{(1)}$ satisfies the weak formulation
\begin{align*}
	\int_{\Omega_1} \left(\nabla e_{\alpha}^{(1)} \cdot \nabla\phi + \gamma e_{\alpha}^{(1)} \phi \right) \mathrm{d}\bfx+ \kappa \langle \Tr_{\Sigma}(e_{\alpha}^{(1)}) , \phi\rangle_{\Sigma}
	= \alpha \langle \nabla u_{\alpha}^{(2)} \cdot \boldsymbol{n}_1,  \phi \rangle_{\Sigma},
\end{align*}
for any $\phi \in H^1(\Omega_1)$. 
Choosing $\phi = e_{\alpha}^{(1)}$, we get
\begin{align}
	\label{error-est-1}
	\int_{\Omega_1} \left( \left|\nabla e_{\alpha}^{(1)}\right|^2 + \gamma \left|e_{\alpha}^{(1)}\right|^2 \right) \mathrm{d}\bfx
	+ \int_{\Sigma} \kappa \left|\Tr_{\Sigma}(e_{\alpha}^{(1)}) \right|^2 \mathrm{d}S
	= \alpha \langle \nabla u_{\alpha}^{(2)} \cdot \boldsymbol{n}_1,  e_{\alpha}^{(1)} \rangle_{\Sigma}. 
\end{align}
{On the other hand, since $u_{\alpha}^{(2)} \in \cW$ and $\Tr_{\Sigma}(e_{\alpha}^{(1)}) \in H^{\half}(\Sigma)$,} we apply Lemma~\ref{lem:tech-1} to obtain  
\begin{equation}
	\label{error-est-2}
    \begin{split}
	\left| \langle  \boldsymbol{n}_1 \cdot \nabla u_{\alpha}^{(2)} ,  e_{\alpha}^{(1)} \rangle_{\Sigma} \right| 
    &\leq C \| u_{\alpha}^{(2)} \|_{H^1(\Omega_2)}  \| \Tr_{\Sigma}(e_{\alpha}^{(1)}) \|_{H^{\half}(\Sigma)} \\
    &\leq C \| u_{\alpha}^{(2)} \|_{H^1(\Omega_2)}  \| e_{\alpha}^{(1)} \|_{H^1(\Omega_1)}.	    
	\end{split}
\end{equation}
Let $\omega \coloneqq \min\{ \gamma, 1 \} > 0$, 
then \eqref{error-est-1} and \eqref{error-est-2} yield
\begin{align*}
	\omega \| e_{\alpha}^{(1)} \|_{H^1(\Omega_1)}^2 
	\leq \alpha C \| u_{\alpha}^{(2)} \|_{H^1(\Omega_2)}  \| e_{\alpha}^{(1)} \|_{H^1(\Omega_1)},
\end{align*}
which implies that
\begin{align}\label{error-est-3}
	\| e_{\alpha}^{(1)} \|_{H^1(\Omega_1)} \leq \alpha \frac{C}{\omega} \| u_{\alpha}^{(2)} \|_{H^1(\Omega_2)}
	< C\alpha.  
\end{align}
Since $e_{\alpha}^{(2)} = \cL(e_{\alpha}^{(1)})$,
combining \eqref{error-est-3} with the boundedness of $\cL$, we get
    \begin{align}\label{error-est-4}
\| e_{\alpha}^{(2)} \|_{H^1(\Omega_2)} < C\alpha.
    \end{align}
Putting together \eqref{error-est-3} and \eqref{error-est-4}, there exists a constant $C$ independent of $\alpha$ such that  
\begin{align*}
	\| u_{\alpha} - u_0 \|_{H^1(\Omega)}
    = \| e_{\alpha} \|_{H^1(\Omega)} \leq C \alpha, 	
\end{align*}
for each $\alpha\in (0, 1)$, 
which implies that 
$u_{\alpha} \to u_0$ strongly in $H^1(\Omega)$ as $\alpha \to 0$,
and {the error satisfies an $O(\alpha)$ upper bound. }
\end{proof}

\section{Numerical Simulations in 1D} 
    \label{sec:numerical}

To verify the order of convergence for the one-dimensional {sharp-interface problem, we manufacture an exact solution to Problem~{\hyperlink{po}{$(P_0)$}} using fixed parameters and solve \eqref{eqn:alpha-prob-1D} for a sequence of $\alpha \in (0,1)$. Using a second-order cell-centered finite difference discretization, we compute the error relative to the manufactured limiting solution, as well as the Cauchy differences between solutions corresponding to successive values of $\alpha$.}

For our 1D problem, we redefine the domains:
 {\begin{align*}
\Omega\coloneqq(-1,1), \quad \Omega_1=\Omega_R\coloneqq(0,1), \quad \Omega_2=\Omega_L\coloneqq(-1,0). 
    \end{align*}
    In this one-dimensional experiment, the geometry is a slight modification of the setting used in the preceding analysis: the interface is $\Sigma = \{0\}, $ the outer boundary is $\Gamma_{\text{out}} = \{-1, 1\}, $ and $\Omega_1 = (0, 1)$ touches $\partial \Omega$ at $x=1$. It retains the Robin interface coupling and the exterior Helmholtz extension while allowing an explicit solution, so it tests the predicted $\alpha$-scaling in this variant. }
   
   {We choose the function
    \begin{align}
    \label{true-sln-1D}
u_0(x) = 
    \begin{cases}	 
u^{(R)}_0(x) = \left(2 -\dfrac{6\pi}{3\pi+1}x^2\right)\left(1 + \sin{(6\pi x)}\right) ,&  \text{ if }  x\in\Omega_R,
    \\
u^{(L)}_0(x) = \dfrac{2\cosh{(x+1)}}{\cosh{(1)}}  ,&  \text{ if } x\in \Omega_L, 
    \end{cases}
    \end{align}}
which is an exact solution to the slightly simplified one-dimensional problem
{\begin{equation}
    \tag{$P_0-\mathrm{1D}$} 
    \begin{aligned}
-\frac{\diff^2 u^{(R)}_0}{\diff x^2} + \gamma u^{(R)}_0 &= q, \quad &&\mbox{in} \  {\Omega_R},
    \\
-\frac{\diff^2 u^{(L)}_0}{\diff x^2} + \beta u^{(L)}_0 &= 0, \quad &&\mbox{in} \  {\Omega_L},
    \\
u_0^{(R)} &= u_0^{(L)}, \quad &&\text{at} \ x=0,
    \\
\frac{\diff u_0^{(R)}}{\diff x}  & = \kappa u_0^{(R)} + g, \quad &&\text{at} \ x=0,
    \\
    \frac{\diff u_0^{(L)}}{\diff x} & = 0,\quad &&\text{at} \ x=-1 \\
\frac{\diff u_0^{(R)}}{\diff x}  & = 0, \quad &&\text{at} \ x= 1,
    \end{aligned}
    \end{equation}
where
    \[
\label{data1D} \beta  = 1, \quad \gamma = 2, \quad \kappa = 3, \quad g = 12\pi-6.  
    \]
    The forcing is chosen from the manufactured right-hand solution: 
    \[
    q(x):=-\frac{\diff^2 u_0^{(R)}}{\diff x^2} + 2 u_0^{(R)}(x).
    \]
    Equivalently, with 
    \[
    a := \frac{6\pi}{3\pi+1}, \quad k:=6\pi, 
    \]
    we have 
   \[
q(x) = \left[ 4 + 2a(1 - x^2) \right] \big( 1 + \sin(kx) \big) + 4akx \cos(kx) + k^2(2 - ax^2) \sin(kx)
\]}
The functions $u^L_0(x)$ and $u^R_0(x)$ are continuous in $\Omega_L$ and $\Omega_R$, respectively, and match at $x=0$, since $u^{(L)}_0(0)=u^{(R)}_0(0)=2$, ensuring that $u_0(x)$ is globally continuous across $\Omega$. 

The one-dimensional $\alpha$-regularized problem is
{\begin{equation}\tag{$P_\alpha - \mathrm{1D}$}\label{eqn:alpha-prob-1D}
\begin{aligned}
    -\dfrac{\mathrm{d}^2 u_\alpha^{(R)}}{\mathrm{d}x^2} + \gamma u_\alpha^{(R)} &= q, && \text{in } \Omega_R, \\[10pt]
    -\alpha \dfrac{\mathrm{d}^2 u_\alpha^{(L)}}{\mathrm{d}x^2} + \alpha\beta u_\alpha^{(L)} &= 0, && \text{in } \Omega_L, \\[10pt]
    u_\alpha^{(R)} &= u_\alpha^{(L)}, && \text{at } x = 0, \\[10pt]
    \dfrac{\mathrm{d}u_\alpha^{(R)}}{\mathrm{d}x} - \alpha \dfrac{\mathrm{d}u_\alpha^{(L)}}{\mathrm{d}x} &= \kappa u_\alpha^{(R)} + g, && \text{at } x = 0, \\[10pt]
    \dfrac{\mathrm{d}u_\alpha^{(L)}}{\mathrm{d}x} &= 0, && \text{at } x = -1, \\[10pt]
    \dfrac{\mathrm{d}u_\alpha^{(R)}}{\mathrm{d}x} &= 0, && \text{at } x = 1.
\end{aligned}
\end{equation}}
To numerically solve Problem \eqref{eqn:alpha-prob-1D} for $\alpha > 0$, we implement a second-order cell-centered finite difference approximation, {with the interface $\{x = 0\}$ at a cell edge. 
The solution is continuous at the interface and the one-sided derivatives satisfy the prescribed
weighted flux jump.}  
The behavior of the numerical solution $u_\alpha$ as it approaches the limit $u_0$ is visualized in Figure~\ref{fig:Sol_WholeDomain}.

{For even $N$, the mesh width is $h=2/N$, with $N/2$ cell centers on each subdomain. Figure~\ref{fig:cell_grid} illustrates this mesh and the four values used to enforce the interface equations. Centered second differences are used in each subdomain, with reflected ghost values at the outer Neumann boundaries. Writing $u_-^L,u_+^L$ for the left physical and ghost values at $-h/2,h/2$, and $u_-^R,u_+^R$ for the right ghost and physical values at $-h/2,h/2$, the interface equations are \[
\frac{u_-^L+u_+^L}{2}=\frac{u_-^R+u_+^R}{2},\qquad
\frac{u_+^R-u_-^R}{h}-\alpha\frac{u_+^L-u_-^L}{h}
=\kappa\frac{u_-^R+u_+^R}{2}+g.
\]
The exterior differential equation is divided by $\alpha>0$ before assembly. The implementation uses the forcing displayed above. In the convergence study, $\alpha_j=10^{-5+(j-1)/6}$, $j=1,\ldots,25$. The error fits use $j=16,\ldots,25$, and the Cauchy fits use pairs $(j,j+1)$ with $j=15,\ldots,24$.}

\begin{figure}[!htbp]
\centering
\begin{tikzpicture}[x=4.2cm,y=1.5cm]
    \draw[thick,->] (-1.4,0) -- (1.4,0) node[right] {$x$};
    \draw (-1,0.1) -- (-1,-0.1) node[below] {$-1$};
    \draw (1,0.1) -- (1,-0.1) node[below] {$1$};
    \draw[dashed,thick] (0,1.2) -- (0,-1.2)
        node[below,font=\small] {Interface $\Sigma$ ($x=0$)};

    \draw[blue,thick] (-1,0.6) -- (0,0.6);
    \draw[blue,dashed,thick] (-1.2,0.6) -- (-1,0.6);
    \draw[blue,dashed,thick] (0,0.6) -- (0.2,0.6);
    \foreach \bx in {-1,-0.8,-0.6,-0.4,-0.2,0} {
        \draw[blue,thick] (\bx,0.5) -- (\bx,0.7);
    }
    \foreach \bx in {-1.2,0.2} {
        \draw[blue,dashed,thick] (\bx,0.5) -- (\bx,0.7);
    }
    \foreach \cx in {-0.9,-0.7,-0.5,-0.3,-0.1} {
        \filldraw[blue] (\cx,0.6) circle (2pt);
    }
    \foreach \cx in {-1.1,0.1} {
        \filldraw[blue,fill=white,thick] (\cx,0.6) circle (2pt);
    }
    \node[blue,left,font=\footnotesize] at (-1.2,0.6) {$\Omega_L$};
    \node[blue,above=3pt,font=\footnotesize] at (-0.1,0.6) {$u_-^L$};
    \node[blue,above=3pt,font=\footnotesize] at (0.1,0.6) {$u_+^L$};

    \draw[red,thick] (0,-0.6) -- (1,-0.6);
    \draw[red,dashed,thick] (-0.2,-0.6) -- (0,-0.6);
    \draw[red,dashed,thick] (1,-0.6) -- (1.2,-0.6);
    \foreach \bx in {0,0.2,0.4,0.6,0.8,1} {
        \draw[red,thick] (\bx,-0.5) -- (\bx,-0.7);
    }
    \foreach \bx in {-0.2,1.2} {
        \draw[red,dashed,thick] (\bx,-0.5) -- (\bx,-0.7);
    }
    \foreach \cx in {0.1,0.3,0.5,0.7,0.9} {
        \filldraw[red] (\cx,-0.6) circle (2pt);
    }
    \foreach \cx in {-0.1,1.1} {
        \filldraw[red,fill=white,thick] (\cx,-0.6) circle (2pt);
    }
    \node[red,right,font=\footnotesize] at (1.2,-0.6) {$\Omega_R$};
    \node[red,below=3pt,font=\footnotesize] at (-0.1,-0.6) {$u_-^R$};
    \node[red,below=3pt,font=\footnotesize] at (0.1,-0.6) {$u_+^R$};
\end{tikzpicture}
\caption{Schematic cell-centered mesh for the interface equations. The left and right grids are offset vertically for clarity. Filled circles denote physical cell centers and open circles denote ghost cells. The four labeled values at $x=\pm h/2$ enter the continuity and weighted flux-jump equations at $x=0$. Outer ghost cells enforce the Neumann conditions at $x=\pm1$.}
\label{fig:cell_grid}
\end{figure}

To quantify the convergence rate with respect to the parameter $\alpha$, define 
{\[
e_\alpha := u_\alpha - u_0, \quad d_k := u_{\alpha_{k+1}} - u_{\alpha_k}.
\]}
{Here, \(e_\alpha\) denotes the error relative to the limiting solution \(u_0\), while \(d_k\) denotes the difference between solutions at successive parameter values \(\alpha_k\) and \(\alpha_{k+1}\), allowing convergence to be assessed without knowing \(u_0\).}
{For a vector $z=(z_i)$ of the values of the physical center of the cell, the discrete norms are
\[
\|z\|_{2,h}=\left(h\sum_i|z_i|^2\right)^{1/2},\qquad
\|z\|_{\infty,h}=\max_i|z_i|.
\]
For the reconstructed $H^1$ errors, separate not-a-knot cubic splines $S_h^L$ and $S_h^R$ are fitted to the numerical values on each subdomain and extrapolated over the endpoint half-cells. The reported quantity is
\[
\left[\sum_{I=L,R}\int_{\Omega_I}
\left(|S_h^I-u_0^I|^2+|(S_h^I)'-(u_0^I)'|^2\right)\,\mathrm{d}x\right]^{1/2},
\]
evaluated by four-point Gauss quadrature on each mesh cell. For Cauchy differences, the same reconstruction is applied to the difference vector and compared with zero. This is a piecewise reconstructed $H^1$ measure; the separate splines are not assumed to have exactly matching interface traces. Figure~\ref{fig:Convergence_Alpha} displays these results on a log-log scale.

The spatial errors are checked against an independent continuous solution for each $\alpha$. With $\lambda=\sqrt2$ and \[
A_\alpha=-\frac{2\alpha\tanh(1)}{\lambda\sinh(\lambda)+(3+\alpha\tanh(1))\cosh(\lambda)},
\] that solution is \begin{equation}\label{finite-alpha-oracle}
\begin{aligned}
u_\alpha^{(R)}(x)&=u_0^{(R)}(x)+A_\alpha\cosh(\lambda(1-x)),\\
u_\alpha^{(L)}(x)&=(2+A_\alpha\cosh\lambda)\frac{\cosh(x+1)}{\cosh1}.
\end{aligned}
\end{equation}
Substitution verifies both differential equations, the interface conditions and the outer Neumann conditions. A separate refinement check at $\alpha=0.2$ with $N=250$, $500$, $1000$, and $2000$ gives fitted spatial orders $2.0003$ in both discrete $L^2$ and $L^\infty$. The accompanying numerical supplement supplies the five MATLAB source files, the convergence data, and the validation results used to reproduce the two figures.
}

As shown in the left panel of Figure~\ref{fig:Convergence_Alpha}, the absolute error scales linearly with $\alpha$ {throughout the plotted range. A least-squares fit over the ten largest values of $\alpha$ gives slopes $0.9956$ in the $L^2$, $L^\infty$, and $H^1$ norms. The corresponding Cauchy-difference slopes are $0.9925$ in all three norms. At the lowest value $\alpha = 10^{-5}$, the spatial discretization errors are only $3.0\%$, $3.4\%$, and $6.4\%$ of the plotted $L^2$, $L^\infty$, and $H^1$ errors, respectively. Thus, the refined mesh resolves the smallest-$\alpha$ signal, and the computations support the first-order convergence $O(\alpha)$ proved above.
}


\begin{figure}[htb!]
	\centering
	\includegraphics[width=\linewidth]{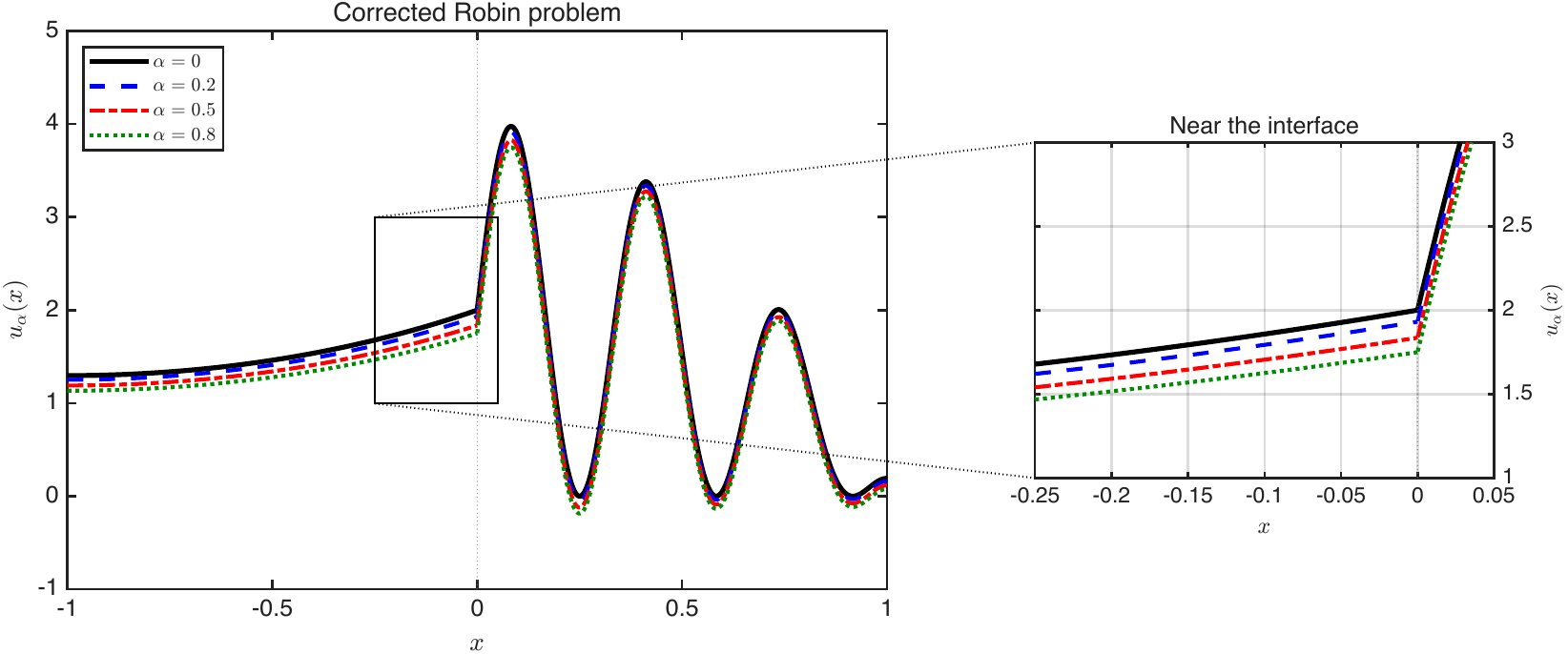}
	\caption{The manufactured limiting solution $u_0$ ($\alpha = 0$) and cell-centered finite difference approximations $u_\alpha$ for $\alpha \in \{0.2, 0.5, 0.8\}$, computed with $N = 10{,}000$ cells for the corrected Robin transmission condition. The zoomed view shows the separation among the curves near the interface $x = 0$.
    }
	\label{fig:Sol_WholeDomain}
\end{figure}

\begin{figure}[htb!]
	\centering
	\includegraphics[width=\linewidth]{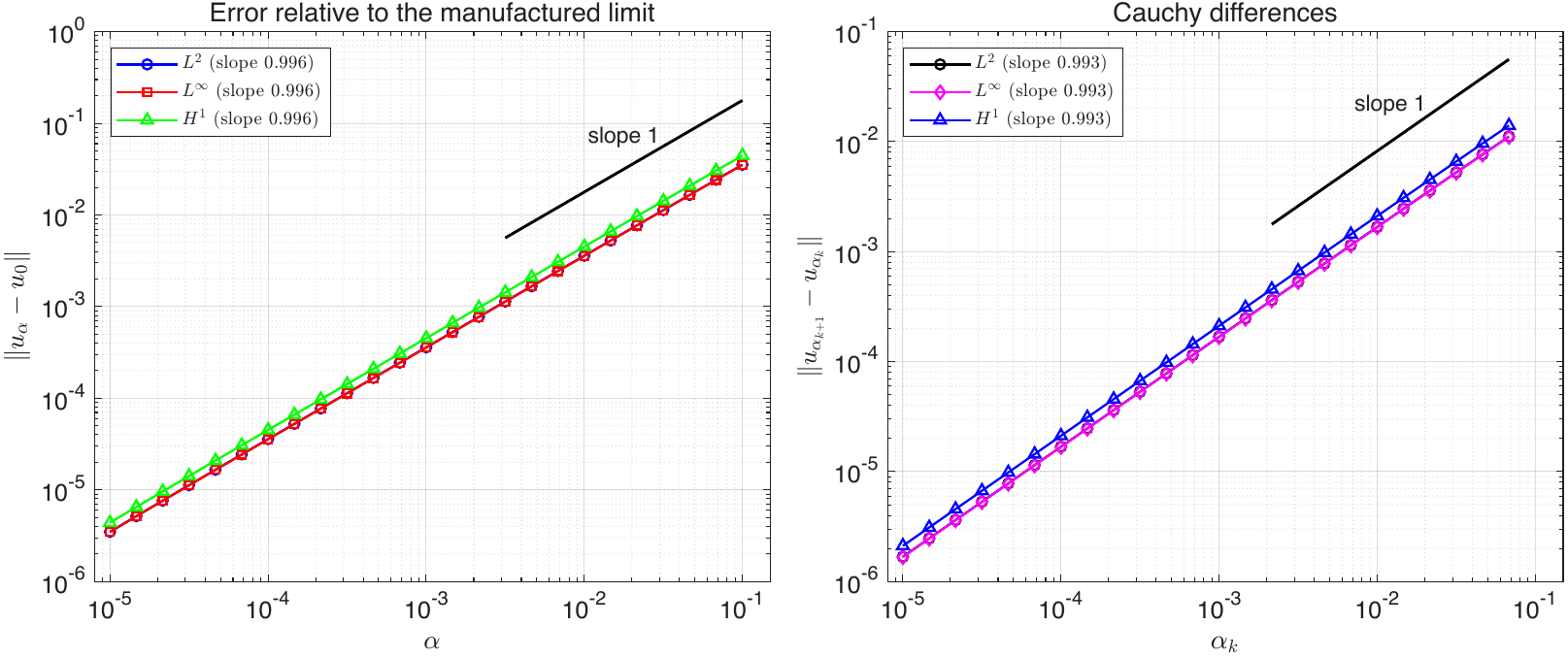}
		\caption{Convergence as $\alpha\to0$ for $\beta=1$, $\gamma=2$, $\kappa=3$, and $g=12\pi-6$, using $N=100,000$ cells and $25$ logarithmically spaced values $\alpha\in[10^{-5},10^{-1}]$. Left: $L^2$, $L^\infty$, and reconstructed $H^1$ errors relative to the  manufactured limiting  solution $u_0$. Right: the corresponding Cauchy differences. Fits over the index ranges stated in the text give slopes approximately $0.996$ for the absolute errors and $0.993$ for  the Cauchy differences;  $H^1$ denotes the piecewise spline reconstruction described there.}
	\label{fig:Convergence_Alpha}
\end{figure}

   \section{Conclusions}
In this paper, we studied the $\alpha$-Limit Problem associated with a family of (almost) degenerate elliptic interface transmission problems arising naturally in the analysis of the Two-Parameter Diffuse Domain Method (DDM2p). Although the $\alpha$-regularized problem admits a standard variational formulation for every $\alpha>0$, the singular limit $\alpha\searrow 0$ produces an interface system whose variational structure is non-standard and subtle. We show that the limiting problem is not naturally posed over the full Sobolev space $H^1(\Omega)$ but, rather, over a closed Hilbert subspace, $\mathcal{H}$, defined by an auxiliary Helmholtz extension problem on the exterior (annular) domain. After identifying the correct limiting energy functional, we establish $\Gamma$-convergence of the parameterized energies in the strong $L^2(\Omega)$ topology and prove convergence of minimizers to the unique minimizer associated with the limiting problem. Using the associated Euler–Lagrange equations, we strengthen this result by proving strong convergence in $H^1(\Omega)$ and derive a convergence rate of order $O(\alpha)$. Numerical experiments in one dimension support the theoretical predictions and suggest that the predicted convergence rate is sharp. This work resolves the $\alpha$-Limit Problem and establishes the target limiting structures needed for a broader program studying simultaneous asymptotic limits in DDM2p approximations.

The analysis developed in this paper provides the foundation for a new methodology/philosophy for approximating solutions to elliptic boundary value problems on complicated domains via the diffuse domain technique. Our method, DDM2p, 
as the name suggests, employs two regularization parameters in constructing the diffuse domain approximation problem. Specifically, it relies on a step that was not explicitly present in previous formulations of the diffuse domain method, namely, $\alpha$-regularization. This feature is important because it leads the analyst to explicitly consider and confront what problem is being approximated in the annular domain $\Omega_2$ that buffers the complicated domain of interest, $\Omega_1$. In their computations, most, if not all, previous works have in practice used the $\alpha$-regularization step -- typically by setting $\alpha$ to be some fixed positive number much smaller than one -- but have not explicitly addressed the rigorous mathematical analysis associated with this choice. This step is necessary because otherwise the resulting stiffness matrices contain diagonal elements that become exponentially small. As we have indicated, the identification of the $\alpha$-regularization step and the rigorous analysis of the $\alpha$-Limit Problem give us the proper framework within which to analyze our fully practical Two-Parameter Diffuse Domain Method. These analyses will be the subject of a series of forthcoming papers.

   \bibliographystyle{plain}
   \bibliography{GammaAlpha}

    \end{document}